\documentclass[12pt]{article}
\usepackage{amsmath,amsfonts,amsthm,amssymb, mathtools}
\usepackage{mathrsfs,graphicx,color,latexsym, tikz,calc,cite,enumerate,indentfirst}
\usetikzlibrary{shadows}
\usepackage{tikz-cd}
\usetikzlibrary{patterns,arrows,decorations.pathreplacing}
\usepackage[colorlinks=true,
linkcolor=blue,citecolor=blue,
urlcolor=blue]{hyperref}

\newtheorem{theorem}{Theorem}[section]
\newtheorem{lemma}{Lemma}[section]
\newtheorem{problem}{Problem}[section]

\newtheorem{corollary}{Corollary}[section]

\newtheorem{prop}{Proposition}[section]

\renewcommand\proofname{\it{Proof}}

\allowdisplaybreaks

\title{\bf Fourier analysis on distance-regular Cayley graphs over abelian groups}

\author{Xiongfeng Zhan$^a$,  Xueyi Huang$^{a,}$\thanks{Corresponding author.},   Lu Lu$^b$ \setcounter{footnote}{-1}\footnote{\emph{Email address:} zhanxfmath@163.com (X. Zhan), huangxy@ecust.edu.cn (X. Huang), lulumath@csu.edu.cn (L. Lu).}\\[2mm]
	\small $^a$School of Mathematics, East China University of Science and Technology, \\
	\small  Shanghai 200237, P. R. China\\
	\small $^b$School of Mathematics and Statistics, Central South University,\\
	\small Changsha, Hunan, 410083, P. R. China
}

\date{}

\begin{document}
	\maketitle
	\begin{abstract}

The problem of constructing or characterizing strongly regular Cayley graphs (or equivalently, regular partial difference sets) has garnered significant attention over the past half-century. In  2003,  Miklavi\v{c} and Poto\v{c}nik [European J. Combin. 24 (2003) 777--784] expanded upon this field by achieving a complete characterization of distance-regular Cayley graphs over cyclic groups through the method of Schur rings. Building on this work, Miklavi\v{c} and Poto\v{c}nik [J. Combin. Theory Ser. B 97 (2007) 14--33] formally proposed the problem of characterizing distance-regular Cayley graphs for arbitrary classes of groups. Within this framework, abelian groups hold particular significance, as numerous distance-regular graphs with classical parameters are precisely Cayley graphs over abelian groups.  In this paper, we employ Fourier analysis on abelian groups to establish connections between distance-regular Cayley graphs over abelian groups and combinatorial objects in finite geometry. By combining these insights with classical results from finite geometry, we classify all distance-regular Cayley graphs over the group  $\mathbb{Z}_n \oplus \mathbb{Z}_p$, where $p$ is an odd prime.
		
		\par\vspace{2mm}
		
		\noindent{\bfseries Keywords:} Distance-regular graph, Cayley graph, Schur ring,  Fourier analysis, finite geometry
		\par\vspace{1mm}
		
		\noindent{\bfseries 2010 MSC:} 05E30, 05C25, 05C50
	\end{abstract}

	\section{Introduction}\label{section::1}

	In graph theory, distance-regular graphs form a class of regular graphs with strong combinatorial symmetry. A connected graph $\Gamma$ is distance-regular if for every vertex $x$ of $\Gamma$, the distance partition of $\Gamma$ with respect to $x$ forms an equitable partition, and all such equitable partitions share a common quotient matrix. While this defining condition is purely combinatorial, the concept of distance-regular graphs holds fundamental importance in design theory and coding theory. Furthermore, it exhibits deep connections with diverse mathematical disciplines including finite group theory, finite geometry, representation theory, and association schemes \cite{BV22,DKT16}.

	Within the study of distance-regular graphs, the characterization and construction of graphs with specific types or parameters constitute essential research problems. Cayley graphs --- vertex-transitive graphs defined through groups and their subsets --- emerge as natural candidates for such investigations. This relevance stems from two observations: most known distance-regular graphs are vertex-transitive \cite{DK06}, and numerous infinite families of strongly regular graphs (the diameter-$2$ case of distance-regular graphs) arise from Cayley graph constructions  \cite{CTZ11,FMX15,FX12,GXY13,LM95,LM05,M84,M94,M89,Mom13,Mom14,Mom17,MX14,MX18,MX22,QCM24,TPF10}.

	Let $G$ be a finite group with identity $e$, and let $S$ be an inverse closed subset of $G\setminus\{e\}$. The \textit{Cayley graph} $\mathrm{Cay}(G,S)$ is defined as the graph with vertex set $G$, where two vertices $g$ and $h$ are adjacent if and only if $g^{-1}h\in S$. The set $S$ is referred to as the  \textit{connection set} of $\mathrm{Cay}(G,S)$. It is well-known that $\mathrm{Cay}(G,S)$ is connected if and only if  $\langle S \rangle=G$, and that $G$ acts regularly on the vertex set of $\mathrm{Cay}(G,S)$ via left multiplication. 
	
	In 2007, Miklavi\v{c} and Poto\v{c}nik  \cite{MP07} (see also \cite[Problem 71]{DKT16})  initiated the systematic study of characterizing distance-regular Cayley graphs by posing the following fundamental problem:
	
	\begin{problem}\label{prob::main}
		For a class of groups $\mathcal{G}$, determine all distance-regular graphs, which are Cayley graphs on a group in $\mathcal{G}$.
	\end{problem}
	
	Early progress on Problem \ref{prob::main} was made by Miklavi\v{c} and Poto\v{c}nik \cite{MP03}, who classified distance-regular Cayley graphs over cyclic groups (known as \textit{circulants}) using the framework of Schur rings.
	
	\begin{theorem}[{\cite[Theorem 1.2, Corollary 3.7]{MP03}}] \label{thm::cir_DRG}
		Let $\Gamma$ be a circulant on $n$ vertices. Then $\Gamma$ is distance-regular if and only if it is isomorphic to one of the following graphs:
		\begin{enumerate}[(i)]\setlength{\itemsep}{0pt}
			\item the cycle $C_n$;
			\item the complete graph $K_n$;
			\item the complete multipartite graph $K_{t\times m}$, where $tm=n$;
			\item the complete bipartite graph without a perfect matching $K_{m,m}-mK_2$, where $2m = n$ and $m$ is odd;
			\item the Paley graph $P(n)$, where $n\equiv 1\pmod 4$  is prime.
		\end{enumerate}
		In particular, $\Gamma$ is a primitive distance-regular graph if and only if $\Gamma\cong K_n$, or $n$ is prime, and $\Gamma\cong C_n$ or $P(n)$.
	\end{theorem}
	
	 Subsequently, Miklavi\v{c} and Poto\v{c}nik \cite{MP03} extended their approach by combining Schur rings with Fourier analysis to characterize distance-regular Cayley graphs over dihedral groups through difference sets \cite{MP07}. Further advancements were achieved by Miklavi\v{c} and \v{S}parl \cite{MS14,MS20}, who employed elementary group theory and structural analysis to classify distance-regular Cayley graphs over abelian groups and generalized dihedral groups under minimality conditions on their connection sets. Significant contributions include the work of Abdollahi, van Dam, and Jazaeri \cite{ADJ17}, who classified distance-regular Cayley graphs of diameter at most $3$ with least eigenvalue $-2$. van Dam and Jazaeri \cite{DJ19,DJ21} later determined some distance-regular Cayley graphs with small valency and provided some characterizations for  bipartite distance-regular Cayley graphs with diameter  $3$ or $4$. For additional results on distance-regular Cayley graphs, including recent developments, we refer to  \cite{HD22,HDL23,HLZ23,ZLH23}.

As is well-known, an effective approach for constructing distance-regular Cayley graphs, particularly strongly regular graphs, involves utilizing Cayley graphs over abelian groups. For instance, numerous infinite families of strongly regular Cayley graphs over the additive group of finite fields have been constructed via methods such as cyclotomic classes \cite{FX12,FMX15}, Gauss sums with even indices \cite{Wu13}, three-valued Gauss periods \cite{Mom17}, and $p$-ary (weakly) regular bent functions \cite{CTZ11,TPF10,QCM24}. Furthermore, it is established that certain distance-regular graphs with classical parameters are precisely distance-regular Cayley graphs over abelian groups. Notable examples include Hamming graph, halved cube, bilinear forms graph, alternating forms graph, Hermitian forms graph, affine $E_6(q)$ graph, and extended ternary Golay code graph (see \cite[p. 194]{BCN89}). However, providing a complete solution to Problem \ref{prob::main} for general abelian groups remains challenging. 
	
In this paper, we investigate distance-regular Cayley graphs over abelian groups with small diameters. We establish necessary conditions for their existence, which are closely connected to finite geometry (see Sections \ref{section::5} and \ref{section::6} for details). Moreover, we demonstrate that these necessary conditions prove particularly useful for the following significant class of abelian groups.

	\begin{problem}\label{prob::main2}
		Let $n$ and $m$ be positive integers with $\mathrm{gcd}(n,m)\neq 1$. 	Characterize all distance-regular Cayley graphs over the group $\mathbb{Z}_n\oplus \mathbb{Z}_m$.
	\end{problem}

	Significant progress has been made toward Problem \ref{prob::main2}. In 2005, Leifman and Muzychuck \cite{LM05} classified strongly regular Cayley graphs over $\mathbb{Z}_{p^s} \oplus \mathbb{Z}_{p^s}$ for odd primes  $p$. Recently, the authors \cite{ZLH23} characterized all distance-regular Cayley graphs over $\mathbb{Z}_{p^s} \oplus \mathbb{Z}_{p}$ and $\mathbb{Z}_{n} \oplus \mathbb{Z}_{2}$ for odd primes $p$.  In this work, we extend these results by providing a complete classification of distance-regular Cayley graphs over  $\mathbb{Z}_{n} \oplus \mathbb{Z}_{p}$, where $p$ is an odd prime. Notably,   $\mathbb{Z}_{n} \oplus \mathbb{Z}_{p}$ becomes cyclic when  $p\nmid n$. Hence, by Theorem \ref{thm::cir_DRG}, we focus exclusively on the case $p\mid n$. Our main result is stated as follows.
	
	\begin{theorem}\label{thm::main}
		Let $p$ be an odd prime, and let $\Gamma$ be a Cayley graph over $\mathbb{Z}_{n}\oplus\mathbb{Z}_{p}$ with $p\mid n$. Then $\Gamma$ is distance-regular if and only if it is isomorphic to one of the following graphs:
		\begin{enumerate}[(i)]\setlength{\itemsep}{0pt}
			\item the complete graph $K_{np}$;
			\item the complete multipartite graph $K_{t\times m}$ with  $tm=np$, which is the complement of the union of $t$ copies of $K_m$;
			\item the complete bipartite graph without a $1$-factor $K_{\frac{np}{2},\frac{np}{2}} - \frac{np}{2}K_2$, where $n\equiv 2\pmod 4$;
			\item the graph $\mathrm{Cay}(\mathbb{Z}_{p}\oplus\mathbb{Z}_{p},S)$ with  $S=\cup_{i=1}^rH_i\setminus\{(0,0)\}$ for some $2\leq r\leq p-1$, where  $H_i$ ($i=1,\ldots,r$) are subgroups of $\mathbb{Z}_{p}\oplus\mathbb{Z}_{p}$ with order $p$.
		\end{enumerate}
		In particular, the graph in (iv) is the line graph of a transversal design $TD(r,p)$, which is a strongly regular graph with parameters $(p^2,r(p-1),p+r^2-3r,r^2-r)$.
	\end{theorem}
	
The paper is organized as follows.  In Section \ref{section::2}, we review fundamental results on association schemes and distance-regular graphs. Section \ref{section::3} presents algebraic characterizations for distance-regular Cayley graphs established by Miklavi\v{c} and Poto\v{c}nik. In Section \ref{section::4}, we introduce key combinatorial objects and classical theorems from finite geometry. Sections \ref{section::5} and \ref{section::6} utilize Fourier analysis on abelian groups to derive necessary conditions for the existence of distance-regular Cayley graphs over abelian groups with small diameter. Finally, Section \ref{section::7} provides a complete proof of Theorem \ref{thm::main}.

	\section{Association schemes and distance-regular graphs}\label{section::2}
	In this section, we introduce some notations and properties related to association schemes and  distance-regular graphs. 
	
	\subsection{Association schemes}\label{subsec::2.1}
	Let $X$ be a finite set, and let  $\mathfrak{R}=\{R_0,R_1,\ldots, R_d\}$ be a set of non-empty subsets of the direct product $X\times X$. Then  $\mathfrak{X}=(X,\mathfrak{R})$ is called an \textit{association scheme of class $d$} if the following  conditions (i)--(iv) hold.
	\begin{enumerate}[(i)]
		\item $R_0=\{(x,x)\mid x\in X\}$.
		\item $\mathfrak{R}=\{R_0, R_1,\ldots, R_d\}$ is  a partition of $X\times X$, that is, $X\times X=R_0\cup R_1\cup\cdots\cup R_d$, and $R_i\cap R_j=\varnothing$ if $i\neq j$.
		\item For each $i\in \{0,\ldots,d\}$, there exists some $i'\in \{0,\ldots,d\}$ such that $R_i^T=R_{i'}$, where $R_i^T=\{(x,y)\mid (y,x)\in R_i\}$.
		\item  Fix $i,j,k\in \{0,\ldots,d\}$. Then the number of elements $z$ of $X$ such that $(x,z)\in R_i$ and $(z,y)\in R_j$ is constant for any $(x,y)\in R_k$. This number is called the \textit{intersection number}, and denoted by $p_{i,j}^k$.
	\end{enumerate}
	Here $R_i$ is called the $i$-th \textit{relation}. Moreover, $\mathfrak{X}=(X,\mathfrak{R})$ is called a \textit{commutative association scheme} if the following condition holds.
	\begin{enumerate}
		\item[(v)] For any $i,j,k\in \{0,\ldots,d\}$, $p_{i,j}^k=p_{j,i}^k$. 
	\end{enumerate}
	Also, $\mathfrak{X}=(X,\mathfrak{R})$ is called a  \textit{symmetric association scheme} if the following condition holds.
	\begin{enumerate}
		\item[(vi)]	For all $i\in \{0,\ldots,d\}$,  $R_i^T=R_i$.
	\end{enumerate}
	It is known that a symmetric association scheme is always commutative.

	Let $\mathfrak{X}=(X,\mathfrak{R})$ be a commutative association scheme, and let $M_X(\mathbb{C})$ be the full matrix algebra of $|X|\times|X|$-matrices over the complex field $\mathbb{C}$ whose rows and columns are indexed by the elements of $X$. For each $i\in\{0,1,\ldots,d\}$, the \textit{adjacency matrix}  $A_i\in M_X(\mathbb{C})$ of the relation $R_i$ is defined as:
	$$
	A_i(x,y)=
	\begin{cases}
		1, &\mbox{if $(x,y)\in R_i$},\\
		0, &\mbox{if $(x,y)\notin R_i$}.\\
	\end{cases}
	$$
	Then by conditions (i)--(v), we obtain the following conditions (I)--(V).
	\begin{enumerate}[(I)]
		\item $A_0=I$, where $I$ is the identity matrix of order $|X|$.
		\item $A_0+A_1+\cdots+A_d=J$, where $J$ is the all-ones matrix of order $|X|$.
		\item For each $i\in \{0,\ldots,d\}$, there exists some $i'\in \{0,\ldots,d\}$ such that $A_i^T=A_{i'}$.
		\item For any $i,j\in \{0,\ldots,d\}$, there exists non-negative integers $p_{i,j}^k$ ($0\leq k\leq d$) such that 
		$$
		A_iA_j=\sum_{k=0}^dp_{i,j}^kA_k.
		$$
		\item For any $i,j\in \{0,\ldots,d\}$, we have $A_iA_j=A_jA_i$.
	\end{enumerate}
	Let $\mathfrak{A}$ be the linear subspace of $M_X(\mathbb{C})$ spanned by the adjacency matrices $A_0,A_1,\ldots,A_d$ of $X$. By (IV) and (V), $\mathfrak{A}$ is a $(d + 1)$-dimensional commutative subalgebra of $M_X(\mathbb{C})$ under the ordinary multiplication. Moreover, by (II), for any $i,j\in\{0,1,\ldots,d\}$, we have
	$$
	A_i\circ A_j=\delta_{i,j}A_i,
	$$
	where `$\circ$' denotes the Hadamard product. This implies that $\mathfrak{A}$ is also a commutative subalgebra of $M_X(\mathbb{C})$  under the Hadamard product. Thus $\mathfrak{A}$ has two algebraic structures, and is  called the \textit{Bose–Mesner algebra}. It is known that $\mathfrak{A}$ is semisimple, and so there exists a basis of primitive idempotents $E_0=\frac{1}{|X|}J, E_1,\ldots, E_d$ in $\mathfrak{A}$. That is, every matrix in $\mathfrak{A}$ can be expressed as a linear combination of $E_0, E_1, \ldots, E_d$, and it holds that $\sum_{i=0}^d E_i = I$ and $E_i E_j = \delta_{ij} E_i$ for all $i, j \in \{0, \ldots, d\}$, where $\delta_{ij} = 1$ if $i = j$, and $\delta_{ij} = 0$ otherwise. This implies the existence of complex numbers $P_i(j) \in \mathbb{C}$ such that
	\[ A_i E_j = P_i(j) E_j \]
	for all $i, j \in \{0, \ldots, d\}$. For any fixed $i\in\{0,1,\ldots,d\}$, the values $P_i(0), P_i(1), \ldots, P_i(d)$ constitute a complete set of the eigenvalues of $A_i$, and furthermore,
	\[
	A_i=\sum_{j=0}^dP_i(j)E_j.
	\]
	On the other hand, since $\mathfrak{A}$ is closed under the  Hadamard product, for any $i,j\in\{0,1,\ldots,d\}$, there exists constants $q_{i,j}^k$ ($0\leq k\leq d$) such that 
	$$
	E_i\circ E_j=\frac{1}{|X|}\sum_{k=0}^d q_{i,j}^kE_k.
	$$
	The structure constants $q_{i,j}^k$ ($0\leq i,j,k\leq d$) of $\mathfrak{A}$ with respect to the Hadamard product are called the  \textit{Krein parameters}. According to \cite[Chapter II, Theorem 3.8]{Bannai}, the Krein parameters $q_{i,j}^k$ are non-negative real numbers. 
	
	\subsection{Schur ring and its duality}

	Let $G$ be a finite group. Given a commutative ring $R$ with identity, the \textit{group algebra} $RG$ of $G$ over $R$ is the set of formal sums $\sum_{g\in G}r_gg$, where $r_g\in R$, equipped with the binary operations
	\begin{equation*}
		\begin{aligned}
			&\sum_{g\in G}r_gg+\sum_{g\in G}s_gg=\sum_{g\in G}(r_g+s_g)g,\\
			&\left(\sum_{g\in G}r_gg\right)\left(\sum_{h\in G}s_hh\right)=\sum_{g,h\in G}(r_gs_h)(gh),\\
		\end{aligned}
	\end{equation*}
	and the scalar multiplication
	\begin{equation*}
		r\sum_{g\in G}r_gg=\sum_{g\in G}(rr_g)g,
	\end{equation*}
	where $r_g,s_g,s_h,r\in R$.
	For a subset $S\subseteq G$, let $\underline{S}$ denote the element $\sum_{s\in S}s$ of $RG$.  In particular, if $S$ contains exactly one element $s$, we write $s$ instead of $\underline{S}$ for simplicity.
	
	Let $\mathbb{Z}$ be the ring of integers, and let $\mathbb{Z}G$ be  the group algebra of $G$ over $\mathbb{Z}$. For an integer $m$ and an element $\sum_{g\in G} r_g g\in \mathbb{Z}G$, we define
	$$\left(\sum_{g\in G} r_g g\right)^{(m)}=\sum_{g\in G} r_g g^m\in \mathbb{Z}G.$$
	Suppose that $\{N_0,N_1,\ldots,N_d\}$ is a partition of $G$ satisfying
	\begin{enumerate}[(i)]
		\item $N_0=\{e\}$;
		\item for any $i\in \{1,\ldots,d\}$, there exists some $j\in \{1,\ldots,d\}$ such that $\underline{N_i}^{(-1)}=\underline{N_j}$;
		\item  for any $i,j\in \{1,\ldots,d\}$, there exist integers $p_{i,j}^k$ ($0\leq k\leq d$) such that $$\underline{N_i}\cdot \underline{N_j}=\sum_{k=0}^rp_{i,j}^k \cdot \underline{N_k}.$$
	\end{enumerate}
	Then the $\mathbb{Z}$-module $\mathcal{S}(G)$ spanned by $\underline{N_0},\underline{N_1},\ldots,\underline{N_d}$ is a subalgebra of $\mathbb{Z}G$, and is called a \textit{Schur ring} over $G$. In this situation, the  basis $\{\underline{N_0},\underline{N_1},\ldots,\underline{N_d}\}$ is called the \textit{simple basis} of the Schur ring $\mathcal{S}(G)$. 	We say that the Schur ring  $\mathcal{S}(G)$ is \textit{primitive} if $\langle N_i\rangle=G$ for every $i\in \{1,\ldots,d\}$. In partitular, if $N_0=\{e\}$ and $N_1=G\setminus \{e\}$, then the Schur ring spanned  by $\underline{N_0}$ and $\underline{N_1}$ is called \textit{trivial}. Clearly, a trivial Schur ring is primitive. 
	
	Now suppose that $G$ is an abelian group. For convenience, we express $G$ as
	\begin{equation*}
		G = \mathbb{Z}_{n_1}\oplus\mathbb{Z}_{n_2}\oplus\cdots\oplus\mathbb{Z}_{n_r},
	\end{equation*}
	where  $n_i$ is a power of some prime number for $1 \leq i \leq r$. It is clear that the order of $G$ is the product $|G| = n_1 n_2 \cdots n_r$. We note that every element $g \in G$ can be uniquely represented as a tuple $g = (g_1, g_2, \ldots, g_r)$, with each $g_i \in \mathbb{Z}_{n_i}$ for $1 \leq i \leq r$. For an element $g \in G$, we define $\chi_g$ as the function from $G$ to  $\mathbb{C}$ by letting
	\begin{equation}\label{equ::char_Ab}
		\chi_g(x) = \prod_{i=1}^r \zeta_{n_i}^{g_i x_i},~\text{for all } x=(x_1,x_2,\ldots,x_r) \in G,
	\end{equation}
	where $\zeta_{n_i}$ denotes a primitive $n_i$-th root of unity. 
	
	Let $\mathcal{S}(G)=\mathrm{span}\{\underline{N_0},\underline{N_1},\ldots,\underline{N_d}\}$ be a Schur ring over the abelian group $G$. For any $i\in\{0,1,\ldots,d\}$, we denote by $R_i=\{(g,h)\mid h^{-1}g\in N_i\}$. Then one can verify that  $\mathfrak{X}=(G,\mathfrak{R}=\{R_0,R_1,\ldots,R_d\})$ is exactly a commutative association scheme (cf. \cite[p. 105]{Bannai}). Moreover, if $N_i$ is inverse closed for any $i\in\{0,1,\ldots,d\}$, then $\mathfrak{X}$ is a symmetric association scheme. The intersection numbers and Krein parameters of $\mathfrak{X}$ are also called the \textit{intersection numbers} and \textit{Krein parameters} of $\mathcal{S}(G)$, respectively. It is known that all Krein parameters of $\mathcal{S}(G)$ are integers. Furthermore, we have the following  classic results about Schur rings over abelian groups.
	
	\begin{lemma}[{\cite[Theorem 3.4]{MP09}}]\label{lem::Schur ring}
		Let $G$ be an abelian group of composite order with at least one cyclic Sylow subgroup. Then there is no non-trivial primitive Schur ring over $G$.
	\end{lemma}

	\begin{lemma}[{\cite[Kochendorfer's theorem]{KT}}]\label{lem::Schur_dq}
		Let $p$ be a prime, and let $a,b$ be positive integers with $a\neq b$. Then there is no non-trivial primitive Schur ring over the group $\mathbb{Z}_{p^{a}}\oplus\mathbb{Z}_{p^{b}}$.
	\end{lemma}

	\begin{lemma}[{\cite[Chapter II, Theorem 6.3]{Bannai}}]\label{lem::dual_Schur_ring}
		Let $G$ be an abelian group, and let $\mathcal{S}(G)=\mathrm{span}\{\underline{N_0},\underline{N_1},$ $\ldots,\underline{N_d}\}$ be a Schur ring over $G$. Let  $\mathcal{R}$ be the equivalence relation on $G$ defined by $g\mathcal{R}h$ if and only if $\chi_g(\underline{N_i})=\chi_h(\underline{N_i})$ for all $i\in \{0,1,\ldots,d\}$. If $E_0,E_1,\ldots,E_f$ are the  equivalence classes of $G$ with respect to $\mathcal{R}$, then $f=d$ and the $\mathbb{Z}$-submodule $\widehat{\mathcal{S}}(G)=\mathrm{span}\{\underline{E_0},\underline{E_1},\ldots,\underline{E_d}\}$ of $\mathbb{Z}G$ is a Schur ring over $G$ with intersection numbers $q_{i,j}^k$, where $q_{i,j}^k$ ($0\leq i,j,k\leq d$) are the Krein parameters of $\mathcal{S}(G)$.
	\end{lemma}
	The Schur ring  $\widehat{\mathcal{S}}(G)$ defined in Lemma \ref{lem::dual_Schur_ring} is called the \textit{dual} of $\mathcal{S}(G)$.

	\subsection{Distance-regular graphs}

	Let $\Gamma$ be a connected graph with vertex set $V(\Gamma)$ and edge set $E(\Gamma)$. The \textit{distance} $\partial_\Gamma(x,y)$ between two vertices $x,y$ of $\Gamma$ is the length of a shortest path connecting them in $\Gamma$, and the \textit{diameter} $d_\Gamma$ of $\Gamma$ is the maximum value of the distances between vertices of $\Gamma$. For $x\in V(\Gamma)$, let $S_i^\Gamma(x)$ denote the set of vertices at distance $i$ from $x$ in $\Gamma$. In particular, we denote $S^\Gamma(x)=S_1^\Gamma(x)$. When $\Gamma$ is clear from the context, we use  $\partial(x,y)$, $d$, $S_i(x)$ and $S(x)$ instead of  $\partial_\Gamma(x,y)$, $d_\Gamma$, $S_i^\Gamma(x)$ and $S^\Gamma(x)$, respectively. For  $x,y\in V(\Gamma)$ with $\partial(x,y)=i$ ($0\leq i\leq d$), let
	$$
	c_i(x,y)=|S_{i-1}(x)\cap S(y)|,~~ a_i(x,y)=|S_{i}(x)\cap S(y)|, ~~ b_i(x,y)=|S_{i+1}(x)\cap S(y)|.
	$$
	Here  $c_0(x,y)=b_d(x,y)=0$. The graph $\Gamma$ is called \textit{distance-regular} if  $c_i(x,y)$, $b_i(x,y)$ and $a_i(x,y)$ only depend on the distance $i$ between $x$ and $y$ but not the choice of $x,y$.
	
	For a distance-regular graph $\Gamma$ with diameter $d$, we denote  $c_i=c_i(x,y)$, $a_i=a_i(x,y)$ and $b_i=b_i(x,y)$, where $x,y\in V(\Gamma)$ and $\partial(x,y)=i$. Note that $c_0=b_d=0$, $a_0=0$ and $c_1=1$. Also, we set $k_i=|S_i(x)|$, where $x\in V(\Gamma)$. Clearly, $k_i$ is independent of the choice of $x$. By definition, $\Gamma$ is a regular graph with valency $k=b_0$, and
	$a_i+b_i+c_i=k$ for $0\leq i\leq d$.  The array $\{b_0,b_1,\ldots,b_{d-1};c_1,c_2,\ldots,c_d\}$ is called the \textit{intersection array} of $\Gamma$. In particular, $\lambda=a_1$ is the number of common neighbors between two adjacent vertices in $\Gamma$, and $\mu=c_2$ is the number of common neighbors between two vertices  at distance $2$ in $\Gamma$.  A distance-regular graph on $n$ vertices with valency $k$ and diameter $2$ is called a  \textit{strongly regular graph} with parameters $(n,k,\lambda=a_1,\mu=c_2)$.

	Suppose that $\Gamma$ is a distance-regular graph of diameter $d$ with vertex set $X=V(\Gamma)$ and edge set $R=E(\Gamma)$. For $0\leq i\leq d$, we define 
	$$
	R_i=\{(x,y)\in X\times X\mid \partial(x,y)=i\}.
	$$
	Then one can verify that the sets $R_i$ ($0\leq i\leq d$)  satisfy the conditions (i)--(vi) in Subsection \ref{subsec::2.1}, and so $\mathfrak{X}=(X,\mathfrak{R}=\{R_0,R_1,\ldots, R_d\})$  is a symmetric association scheme. In this context, the intersection numbers $p_{i,j}^k$ and Krein parameters $q_{i,j}^k$ of $\mathfrak{X}$ are also called the \textit{intersection numbers} and \textit{Krein parameters} of $\Gamma$, respectively. Note that  $p_{1,i+1}^i=b_i$, $p_{1,i}^i=a_i$ and $p_{1,i-1}^i=c_i$ for $0\leq i\leq d$. Additionally, for $i,j,k\in\{0,1,\ldots,d\}$, if $p_{i,j}^k\neq 0$ then $k\leq i+j$, and moreover,  $p_{i,j}^{i+j}\neq 0$.

	A symmetric association scheme together with an ordering of relations  is called \textit{$P$-polynomial}  if $p_{i,j}^k\neq 0$ implies $k\leq i+j$ for all $i,j,k\in\{0,1,\ldots,d\}$, and also $p_{i,j}^{i+j}\neq 0$ for all $i,j\in\{0,1,\ldots,d\}$. By definition, the symmetric association scheme derived from a distance-regular graph is $P$-polynomial. Conversely, every $P$-polynomial association scheme is derived from a distance-regular graph. Therefore, a distance-regular graph is equivalent to a $P$-polynomial association scheme. 
	\begin{lemma}[{\cite[Proposition 2.7.1]{BCN89}}] \label{lem::P-polynomial}
		Let $\mathfrak{X}=(X,\mathfrak{R})$ be a symmetric association scheme with an ordering of relations $R_0, R_1, \ldots, R_d$. Then  $\mathfrak{X}$ is $P$-polynomial if and only if $(X,R_1)$ is a distance-regular graph.
	\end{lemma}
	
	Analogously, a symmetric association scheme together with an ordering of primitive idempotents is called  \textit{$Q$-polynomial} if $q_{i,j}^k\neq 0$ implies $k\leq i+j$ for all $i,j,k\in\{0,1,\ldots,d\}$, and also $q_{i,j}^{i+j}\neq 0$ for all $i,j\in\{0,1,\ldots,d\}$.
	In particular, we say that a distance-regular graph is $Q$-\textit{polynomial} if the symmetric association scheme derived from it is  $Q$-polynomial.

	A $P$-polynomial (resp. $Q$-polynomial) association scheme is called \textit{bipartite} (resp. $Q$-\textit{bipartite}) if $p_{i,j}^k= 0$ (resp. $q_{i,j}^k= 0$) whenever $i+j+k$ is odd. A $P$-polynomial (resp. $Q$-polynomial) association scheme is called \textit{antipodal} (resp. $Q$-\textit{antipodal}) if $p_{d,d}^k= 0$ (resp. $q_{d,d}^k= 0$) whenever $k\notin \{0,d\}$. 
	
	\begin{lemma}[{\cite[p. 241]{BCN89}}]\label{lem::dual imprimitivity}
		Let $\Gamma$ be a $Q$-polynomial distance-regular graph. Then $\Gamma$ is bipartite (resp. antipodal) if and only if  $\Gamma$ is $Q$-antipodal (resp. $Q$-bipartite), that is, the symmetric association scheme derived from $\Gamma$ is $Q$-antipodal (resp. $Q$-bipartite).
	\end{lemma}


	\subsection{Primitivity of distance-regular graphs}
	
	Let $\Gamma$ be a graph, and let $\mathcal{B}=\{B_1,\ldots,B_\ell\}$ be a partition of $V(\Gamma)$ (here $B_i$ are called \textit{blocks}). The \textit{quotient graph} of $\Gamma$ with respect to $\mathcal{B}$, denoted by $\Gamma_\mathcal{B}$, is the graph with vertex set $\mathcal{B}$, and with  $B_i,B_j$ ($i\neq j$) adjacent if and only if there exists at least one edge between  $B_i$ and $B_j$ in $\Gamma$. Moreover, we say that $\mathcal{B}$ is an \textit{equitable partition} of $\Gamma$ if there are integers $b_{ij}$ ($1\leq i,j\leq \ell$) such that every vertex in $B_i$ has exactly $b_{ij}$ neighbors in $B_j$. In particular, if every block of $\mathcal{B}$ is an independent set, and between any two blocks there are either no edges or there is a perfect matching, then $\mathcal{B}$ is an equitable partition of  $\Gamma$. In this situation,  $\Gamma$ is called a \textit{cover} of its quotient graph $\Gamma_\mathcal{B}$, and the blocks are called \textit{fibres}. If $\Gamma_\mathcal{B}$ is connected, then all fibres have the same size, say $r$, called \textit{covering index}.

	A graph $\Gamma$ with diameter $d$ is \textit{antipodal} if the relation $\mathcal{R}$ on $V(\Gamma)$ defined by $x\mathcal{R}y\Leftrightarrow\partial(x,y)\in\{0,d\}$ is an equivalence relation. Under this equivalence relation, the corresponding equivalence classes are called \textit{antipodal classes}. A cover of index $r$, in which the fibres are antipodal classes, is called \textit{an $r$-fold  antipodal cover} of its quotient. In particular, if $\Gamma$ is an antipodal distance-regular graph with diameter $d$, then all antipodal classes have the same size, say $r$, and form an equitable partition $\mathcal{B}^\ast$ of $\Gamma$. The quotient graph $\overline{\Gamma}:=\Gamma_{\mathcal{B}^\ast}$ is called the \textit{antipodal quotient} of $\Gamma$. If $d=2$, then $\Gamma$ is a complete multipartite graph. If $d\geq 3$, then the edges between  two distinct antipodal classes of $\Gamma$ form an empty set or a perfect matching. Thus $\Gamma$ is an $r$-fold antipodal  cover of $\overline{\Gamma}$ with the antipodal classes as its fibres. Moreover, it is known that a distance-regular graph $\Gamma$ with diameter $d$ is antipodal if and only if $b_i=c_{d-i}$ for every $i\neq \lfloor\frac{d}{2}\rfloor$.

	Let $\Gamma$ be a distance-regular graph with diameter $d$. For  $i\in\{1,\ldots,d\}$, the \textit{$i$-th distance graph} $\Gamma_i$  is  the graph with vertex set  $V(\Gamma)$ in which two distinct vertices are adjacent if and only if they are at distance $i$ in $\Gamma$. If, for any $1\le i\le d$, $\Gamma_i$ is connected, then $\Gamma$ is {\it primitive}. Otherwise, $\Gamma$ is {\it imprimitive}.  It is  known that an imprimitive distance-regular graph with valency at least $3$ is either bipartite, antipodal, or both \cite[Theorem 4.2.1]{BCN89}. In particular, if $\Gamma$ is bipartite, then $\Gamma_2$ has two connected components, which are called the \textit{halved graphs} of $\Gamma$ and denoted by $\Gamma^+$ and $\Gamma^-$. For convenience, we use  $\frac{1}{2}\Gamma$ to represent any one of these two graphs. 
	
	\begin{lemma}[{\cite[p. 140, p. 141]{BCN89}}] \label{lem::imprimitive}
		Let $\Gamma$ denote an imprimitive distance-regular graph with diameter $d$ and valency $k \geq 3$. Then the following hold.
		\begin{enumerate}[(i)]\setlength{\itemsep}{0pt}
			\item If $\Gamma$ is bipartite, then the halved graphs of $\Gamma$ are non-bipartite distance-regular graphs with diameter $\lfloor\frac{d}{2}\rfloor$.
			\item If $\Gamma$  is antipodal, then $\overline{\Gamma}$ is a distance-regular graph with diameter $\lfloor\frac{d}{2}\rfloor$.
			\item  If $\Gamma$ is antipodal, then $\overline{\Gamma}$ is not antipodal, except when $d\leq 3$  (in that case $\overline{\Gamma}$  is a complete graph), or when $\Gamma$ is bipartite with $d = 4$ (in that case $\overline{\Gamma}$ is a complete bipartite graph).
			
			\item If $\Gamma$ is antipodal and has odd diameter or is not bipartite, then $\overline{\Gamma}$ is primitive.
			\item If $\Gamma$ is bipartite and has odd diameter or is not antipodal, then the halved graphs of $\Gamma$ are primitive.
			\item If $\Gamma$ has even diameter and is both bipartite and antipodal, then $\overline{\Gamma}$ is bipartite. Moreover, if $\frac{1}{2}\Gamma$ is a halved graph of $\Gamma$, then it is antipodal, and $\overline{\frac{1}{2}\Gamma}$ is primitive and isomorphic to $\frac{1}{2}\overline{\Gamma}$.
		\end{enumerate}
	\end{lemma}

	For distance-regular Cayley graphs over abelian groups, we have following result about antipodal quotients and halved graphs.
	
	\begin{lemma}\label{lem::DRG_dq} 
		Let $G$ be an abelian group, and let $\Gamma$ be a distance-regular Cayley graph over $G$. Then the following two statements hold. 
		\begin{enumerate}[(i)]
			\item If $\Gamma$ is antipodal and $H$ is the antipodal class containing the identity vertex $e$, then $H$ is a subgroup of $G$, and the antipodal quotient of $\Gamma$ is distance-regular and isomorphic to $\mathrm{Cay}(G/H,S/H)$, where $S/H=\{sH\mid s\in S\}$;
			\item If $\Gamma$ is bipartite and $H$ is the bipartition set containing the identity vertex $e$, then $H$ is an index $2$ subgroup of  $G$, and the halved graphs  of $\Gamma$ are distance-regular and isomorphic to $\mathrm{Cay}(H,S_2(e))$.
		\end{enumerate}
	\end{lemma}
	\begin{proof}
		
		(i)	By Lemma \ref{lem::imprimitive}, it suffices to prove that $\overline{\Gamma}\cong\mathrm{Cay}(G/H,S/H)$.	Since $\Gamma$ is antipodal,   the relation $\mathcal{R}$ on $V(\Gamma)$ defined by $x\mathcal{R}y\Leftrightarrow\partial(x,y)\in\{0,d\}\Leftrightarrow\partial(y^{-1}x,e)\in\{0,d\}\Leftrightarrow y^{-1}x\in H$ is an equivalence relation. For any $h_1,h_2\in H$, we have $h_1\mathcal{R}e$ and $e\mathcal{R}h_2$, and hence $h_1\mathcal{R} h_2$, or equivalently,  $h_2^{-1}h_1\in H$. Thus  $H$ is a subgroup of $G$, and  the antipodal classes of $\Gamma$ coincide with the cosets of $H$ in $G$. For any two vertices  $xH$ and $yH$ of $\overline{\Gamma}$, we have that $xH$ and $yH$ are adjacent if and only if there exists some edge between $xH$ and $yH$ in $\Gamma$, which is the case if and only if there exists some $h_1,h_2\in H$ such that $(xh_1)^{-1}yh_2\in S$, which is the case if and only if $(xH)^{-1}yH\in S/H$. Therefore, we conclude that  $\overline{\Gamma}\cong\mathrm{Cay}(G/H,S/H)$, and the result follows.
		
		(ii)	By Lemma \ref{lem::imprimitive}, it suffices to prove that $\Gamma^+\cong \Gamma^-\cong \mathrm{Cay}(H,S_2(e))$. Suppose that  $V(\Gamma^+)=H$.  By a similar way as in (i), we can prove that $H$ is an index $2$ subgroup of $G$. For any two vertices $x,y\in V(\Gamma^+)=H$, we have that $x,y$ are adjacent if and only if $\partial(x,y)=2$, which is the case if and only if $\partial(e,x^{-1}y)=2$, or equivalently, $x^{-1}y\in S_2(e)$. Therefore, we conclude that $\Gamma^+\cong\mathrm{Cay}(H,S_2(e)$. Furthermore, as $\Gamma$ is vertex-transitive, we have  $\Gamma^-\cong \Gamma^+$, and the result follows.
	\end{proof}
	
	\begin{lemma}[{\cite[p. 425, p. 431]{BCN89}}]\label{lem::antipodal_DRG}
		Let $\Gamma$ be an $r$-fold antipodal distance-regular graph on $n$ vertices with  diameter $d$ and valency $k$.
		\begin{enumerate}[(i)]\setlength{\itemsep}{0pt}
			\item If $\Gamma$ is non-bipartite and $d=3$, then $n=r(k+1)$, $k=\mu(r-1)+\lambda+1$, $0<\mu<k-1$ and $\Gamma$ has  the intersection array $\{k,\mu(r-1),1;1,\mu,k\}$ and the spectrum $\{k^1,\theta_1^{m_1},\theta_2^k,\theta_3^{m_3}\}$, where
			\begin{equation*}
				\theta_1=\frac{\lambda-\mu}{2}+\delta,~~\theta_2=-1,~~\theta_3=\frac{\lambda-\mu}{2}-\delta,~~\delta=\sqrt{k+\left(\frac{\lambda-\mu}{2}\right)^2},
			\end{equation*}
			and
			\begin{equation*}
				m_1=-\frac{\theta_3}{\theta_1-\theta_3}(r-1)(k+1),~~m_3=\frac{\theta_1}{\theta_1-\theta_3}(r-1)(k+1).
			\end{equation*}
			Moreover, if $\lambda\neq\mu$, then all eigenvalues  of $\Gamma$ are integers.
			
			\item If $\Gamma$ is bipartite and $d=4$, then $n=2r^2\mu$, $k=r\mu$, and $\Gamma$ has the intersection array $\{r\mu, r\mu-1,(r-1)\mu, 1;1,\mu,r\mu-1,r\mu\}$.
		\end{enumerate}
	\end{lemma}

	A \textit{conference graph} is a strongly regular graph with parameters $(n,k=(n-1)/2,(n-5)/4,\mu=(n-1)/4)$, where $n\equiv 1\pmod 4$.  Let $\mathbb{F}_q$ denote the finite field of order $q$  where $q$ is a prime power and $q\equiv 1\pmod 4$. The \textit{Paley graph} $P(q)$ is defined as the graph with vertex set $\mathbb{F}_q$ in which two distinct vertices $u,v$ are adjacent if and only if $u-v$ is a square in the multiplicative group of $\mathbb{F}_q$. It is known that $P(q)$ is a  conference graph \cite{ER63}.

	\begin{lemma}[{\cite[p. 180]{BCN89}}]\label{lem::confer}
		Let $\Gamma$ be a conference graph (or particularly, Paley graph). Then $\Gamma$ has no distance-regular $r$-fold antipodal covers for $r>1$, except for the pentagon $C_5\cong P(5)$, which is covered by the decagon $C_{10}$. Moreover, $\Gamma$ cannot be a halved graph of a bipartite distance-regular graph.
	\end{lemma}

	The \textit{Hamming graph} $H(n,q)$ is the graph having as vertex set the collection of all $n$-tuples with entries in a fixed set of size $q$, where two $n$-tuples are adjacent when they differ in only one coordinate. Note that $H(2,v)$ is just the lattice graph $K_v\times K_v$, which is the Cartesian product of two copies of $K_v$.
	
	\begin{lemma}[{\cite[Proposition 5.1]{HG}}]\label{lem::anti_Ham}
		Let $n,q\geq2$. Then $H(n,q)$ has no distance-regular $r$-fold antipodal covers for $r>1$, except for $H(2,2)$.
	\end{lemma}

	\section{The algebraic characterizations of distance-regular Cayley graphs}\label{section::3}
	
	In this section, we present several algebraic characterizations for distance-regular Cayley graphs, which were established by Miklavi\v{c} and Poto\v{c}nik in \cite{MP03,MP07}.
	
	\subsection{Schur ring and distance-regular Cayley graphs}

	Let $\Gamma=\mathrm{Cay}(G,S)$ be a connected Cayley graph with diameter $d$. For $i\in\{0,1,\ldots,d\}$, we denote by 
	\begin{equation}\label{equ::DM}
		S_i=\{g\in G\mid \partial_\Gamma(g,e)=i\}.
	\end{equation} 
	Then the  $\mathbb{Z}$-submodule of  $\mathbb{Z}G$ spanned by  $\underline{S_0},\underline{S_1},\ldots,\underline{S_d}$ is called the \textit{distance module} of  $\Gamma$, and is denoted by \textit{$\mathcal{D}_\mathbb{Z}(G,S)$}. 
	
	In \cite{MP03}, Miklavi\v{c} and Poto\v{c}nik  provided an algebraic characterization for  distance-regular Cayley graphs in terms of Schur rings and distance modules.
	
	\begin{lemma}[{\cite[Proposition 3.6]{MP03}}] \label{lem::Schur_DRG}
		Let $\Gamma=\mathrm{Cay}(G,S)$  denote a distance-regular Cayley graph and $\mathcal{D}=\mathcal{D}_\mathbb{Z}(G,S)$ its distance module. Then:
		\begin{enumerate}[(i)]\setlength{\itemsep}{0pt}
			\item $\mathcal{D}$ is a (primitive) Schur ring over $G$ if and only if $\Gamma$ is a (primitive) distance-regular graph;
			\item $\mathcal{D}$ is the trivial Schur ring  over $G$ if and only if $\Gamma$ is isomorphic to the complete graph.
		\end{enumerate}
	\end{lemma}

	Suppose that $p$ is a prime and $p\neq n$.	If $\Gamma$ is a primitive distance-regular Cayley graph over $\mathbb{Z}_{n}\oplus\mathbb{Z}_{p}$, then its distance module would be a primitive Schur ring over $\mathbb{Z}_{n}\oplus\mathbb{Z}_{p}$ by Lemma \ref{lem::Schur_DRG} (i), and hence must be  the trivial Schur ring  by Lemma \ref{lem::Schur ring} and Lemma \ref{lem::Schur_dq}. Therefore, by Lemma \ref{lem::Schur_DRG} (ii), we obtain the following result.
	
	\begin{corollary}\label{cor::pri_DRG}
		Let $\Gamma$ be a primitive distance-regular Cayley graph over $\mathbb{Z}_{n}\oplus\mathbb{Z}_{p}$ where $p$ is a prime and $p\neq n$. Then $\Gamma$ is isomorphic to the complete graph $K_{np}$.
	\end{corollary}

	As  every Cayley graph is vertex-transitive, by the definitions of Cayley graphs and distance-regular graphs, we immediately deduce the following characterization for distance-regular Cayley graphs.
	\begin{lemma}\label{lem::DR}
		Let $\Gamma=\mathrm{Cay}(G,S)$ be a Cayley graph with diameter $d$. Then $\Gamma$ is distance-regular with intersection array $\{b_0,b_1,\ldots,b_{d-1};c_1,c_2,\ldots,c_d\}$ if and only if
		\begin{equation}\label{equ::DR}
			\left\{\begin{aligned}			\underline{S_1}\cdot\underline{S_1}&=b_0\cdot e+a_1\underline{S_1}+c_2\underline{S_2},\\
				\underline{S_2}\cdot\underline{S_1}&=b_1\underline{S_1}+a_2\underline{S_2}+c_3\underline{S_3},\\ 
				&\vdots\\
				\underline{S_d}\cdot\underline{S_1}&=b_{d-1}\underline{S_{d-1}}+a_{d}\underline{S_d},\\ \end{aligned}\right.
		\end{equation}
		where $S_i$ ($1\leq i\leq d$) are defined in \eqref{equ::DM}. 
	\end{lemma}
	
	Since $e+\underline{S_1}+\underline{S_2}+\cdots+\underline{S_d}=\underline{G}$, we see that the conclusion of Lemma \ref{lem::DR} still holds if we remove an arbitrary equation from \eqref{equ::DR}. Recall that a distance-regular graph with diameter $d$ is antipodal if and only if $b_i=c_{d-i}$ for every $i\neq \lfloor\frac{d}{2}\rfloor$. Additionally,  the intersection array of an $r$-antipodal distance-regular graph with  diameter $3$ must be of the form $\{k,k-\lambda-1=\mu(r-1),1;1,\mu,k\}$.  Thus, by Lemma \ref{lem::DR}, we can deduce the following result immediately.
	
	\begin{corollary}\label{cor::DR}
		Let $\Gamma=\mathrm{Cay}(G,S)$ be a  Cayley graph with diameter $3$.
		Then $\Gamma$ is an antipodal distance-regular graph with intersection array $\{k,k-\lambda-1=\mu(r-1),1;1,\mu,k\}$ if and only if 
		\begin{equation*}\label{equ::DR3}
			\left\{\begin{aligned}
				&\underline{S}^2=k\cdot e+(\lambda-\mu)\underline{S}+\mu(\underline{G}-\underline{S_3}-e),\\
				&(\underline{S_3}+e)\cdot(\underline{S}+e)=\underline{G}.\\ \end{aligned}\right.
		\end{equation*}
	\end{corollary}

	\subsection{Distance-regular Cayley graphs over abelian groups}

	Let $G$ be a finite group, and let $V$ be a $d$-dimensional vector space over $\mathbb{C}$. A \textit{representation} of $G$ on $V$ is a group homomorphism $\rho: G \rightarrow GL(V)$, where $GL(V)$ is the general linear group consisting of invertible linear transformations on $V$. The dimension $d$ is referred to as the \textit{degree} of the representation $\rho$. The \textit{trivial representation} of $G$ is the homomorphism $\rho_0: G \rightarrow \mathbb{C}^*$ defined by letting $\rho_0(g) = 1$ for all $g \in G$, where $\mathbb{C}^*$ is the multiplicative group of $\mathbb{C}$.
	
	A subspace $W$ of $V$ is said to be $G$-\textit{invariant} if for every $g \in G$ and $w \in W$, we have $\rho(g)w \in W$. Clearly, both the zero subspace $\{0\}$ and the entire space $V$ are $G$-invariant subspaces. A representation $\rho: G \rightarrow GL(V)$ is said to be \textit{irreducible} if $\{0\}$ and $V$ are the only $G$-invariant subspaces of $V$.
	
	Consider two vector spaces $V$ and $U$ over $\mathbb{C}$ with the same finite dimension. Suppose that $\rho$ and $\tau$ are representations of $G$ on $V$ and $U$, respectively. We say that $\rho$ and $\tau$ are \textit{equivalent}, denoted by $\rho \sim \tau$, if there exists a linear isomorphism $f: V \rightarrow U$ that satisfies the following commutative condition for every $g \in G$:
	\begin{equation*}
		\begin{tikzcd}
			V \arrow[r, "\rho(g)"] \arrow[d, "f"'] & V \arrow[d, "f"] \\
			U \arrow[r, "\tau(g)"'] & U
		\end{tikzcd}
	\end{equation*}

	The \textit{character} associated with a representation $\rho: G \rightarrow GL(V)$ is a function $\chi_\rho: G \rightarrow \mathbb{C}$ that is defined as $\chi_\rho(g) = \text{tr}(\rho(g))$ for all $g \in G$. The \textit{degree} of the character $\chi_\rho$ corresponds to the degree of the representation $\rho$. Clearly, all equivalent representations share the same character. A character $\chi_\rho$ is called \textit{irreducible} if the representation it corresponds to, $\rho$, is itself irreducible. The set of all irreducible characters of a group $G$ is denoted by $\widehat{G}$.

	Note that both a representation $\rho$ and its corresponding character $\chi_\rho$ can be extended linearly to the group algebra $\mathbb{C}G$. Let $G$ be a finite group. For any $\mathcal{K}\in\mathbb{C}G$ and $\chi\in \widehat{G}$, we denote by 
	$$\chi(\mathcal{K})=\sum_{g\in G}a_g(\mathcal{K})\chi(g),$$ 
	where $a_g(\mathcal{K})$ is the coefficient of $g$ in $\mathcal{K}$. Then the \textit{Fourier inversion formula} gives that  
	\begin{equation}\label{equ::FI}
		a_g(\mathcal{K})=\frac{1}{|G|}\sum_{\chi\in \widehat{G}}\chi(\mathcal{K}g^{-1})\chi(e).
	\end{equation}
	For a more comprehensive understanding of representation theory, one may consult the reference \cite{Ste12}.

	Now suppose that $G$ is an abelian group. It is known that the set of all irreducible characters of $G$ is 
	\begin{equation*}
		\widehat{G} = \{\chi_g \mid g \in G\},
	\end{equation*}
	where $\chi_g$ is the function defined in \eqref{equ::char_Ab}.	Furthermore, for any $g \in G$, the character $\chi_g$ is  a group homomorphism from $G$ to $\mathbb{C}^*$, and so is  a representation of $G$ with degree $1$. Let $\mathcal{K},\mathcal{L}\in\mathbb{C}G$. The Fourier inversion formula \eqref{equ::FI} implies that
	\[\mathcal{K}=\mathcal{L}~\mbox{if and only if}~ \chi_g(\mathcal{K})=\chi_g(\mathcal{L})~\mbox{for all}~g \in G.\] 
	Combining this with Lemma \ref{lem::DR}, we 
	obtain the following characterization for distance-regular Cayley graphs over abelian groups.  
	\begin{lemma}\label{lem::DR_1}
		Let $G$ be an abelian group, and let $\Gamma = \mathrm{Cay}(G, S)$ be a Cayley graph with diameter $d$. Then $\Gamma$ is distance-regular with intersection array $\{b_0, b_1, \ldots, b_{d-1}; c_1, c_2, \ldots, c_d\}$ if and only if for every $g \in G$, the following system of equations holds:
		\begin{equation*}\label{equ::CDR}
			\left\{\begin{aligned}
				\chi_g(\underline{S_1})\cdot\chi_g(\underline{S_1})&=b_0+a_1\chi_g(\underline{S_1})+c_2\chi_g(\underline{S_2}),\\
				\chi_g(\underline{S_2})\cdot\chi_g(\underline{S_1})&=b_1\chi_g(\underline{S_1})+a_2\chi_g(\underline{S_2})+c_3\chi_g(\underline{S_3}),\\ 
				&\vdots\\
				\chi_g(\underline{S_d})\cdot\chi_g(\underline{S_1})&=b_{d-1}\chi_g(\underline{S_{d-1}})+a_{d}\chi_g(\underline{S_d}).
			\end{aligned}\right.
		\end{equation*}
	\end{lemma}	
	
	Let $\Gamma = \mathrm{Cay}(G, S)$ be a Cayley graph over the abelian group $G$. According to \cite{Babai}, the eigenvalues of $\Gamma = \mathrm{Cay}(G, S)$ are given by
	\[ \chi_g(\underline{S}) = \sum_{s \in S} \chi_g(s), \quad \text{for all} \ g \in G. \]
	Suppose further that $\Gamma$ is  distance-regular and has diameter $d$. Then $\Gamma$ has exactly $d + 1$ distinct eigenvalues, denoted as $\theta_0 > \theta_1 > \dots > \theta_d$. Let $\mathfrak{A}$ be the Bose-Mesner algebra corresponding to the symmetric association scheme derived from $\Gamma$, and let $E_0 = \frac{1}{|G|}J, E_1, \dots, E_d$ be the primitive idempotents of $\mathfrak{A}$ such that $A(\Gamma)E_i = A_1E_i = \theta_iE_i$ for all $i \in \{0, 1, \dots, d\}$. We denote by $\widehat{S}_i = \{g \in G \mid \chi_g(\underline{S}) = \theta_i\}$. Clearly, there is a one-to-one correspondence between $\widehat{S}_i$ and $E_i$.
	Let $\tau$ be a permutation on $\{0, 1, \dots, d\}$ that fixes $0$. We say that $\Gamma$ has a \textit{$Q$-polynomial ordering} $\widehat{S}_{\tau(0)}, \widehat{S}_{\tau(1)}, \dots, \widehat{S}_{\tau(d)}$ if it is $Q$-polynomial with respect to the ordering of primitive idempotents $E_{\tau(0)}, E_{\tau(1)}, \dots, E_{\tau(d)}$.

	\begin{lemma}\label{lem::dual_graph}
		Let $G$ be an abelian group, and let $\Gamma = \mathrm{Cay}(G, S)$ be a distance-regular Cayley graph with diameter $d$ over $G$. Let $\theta_0 > \theta_1 > \cdots > \theta_d$ be all the distinct eigenvalues of $\Gamma$, and let $\widehat{S}_i = \{g \in G \mid \chi_g(\underline{S}) = \theta_i\}$ for $0 \leq i \leq d$. If $\Gamma$ has a $Q$-polynomial ordering $\widehat{S}_{\tau(0)}, \widehat{S}_{\tau(1)}, \ldots, \widehat{S}_{\tau(d)}$, where $\tau$ is a permutation on the set $\{0, 1, \ldots, d\}$ that fixes $0$, then the Cayley graph  $\widehat{\Gamma} = \mathrm{Cay}(G, \widehat{S}_{\tau(1)})$ is a distance-regular graph of diameter $d$ with intersection numbers $q_{i,j}^k$, where $q_{i,j}^k$ are the Krein parameters of $\Gamma$.
	\end{lemma}
	\begin{proof}
		By Lemma \ref{lem::Schur_DRG}, $\mathcal{S}(G) = \mathrm{span}\{\underline{S_0}, \underline{S_1}, \ldots, \underline{S_d}\}$ is a Schur ring over $G$. Note that $\widehat{S}_i$ is inverse closed for all $i\in\{0,1,\ldots,d\}$. For any $g,h\in G$,  we have $g, h \in \widehat{S}_i$ if and only if  $\chi_g(\underline{S}) = \chi_h(\underline{S}) = \theta_i$, which is the case if and only if $\chi_g(\underline{S_j}) = \chi_h(\underline{S_j})$ for all $j \in \{0, 1, \ldots, d\}$ by Lemma \ref{lem::DR_1}. Then Lemma \ref{lem::dual_Schur_ring} indicates that  $\widehat{\mathcal{S}}(G)=\mathrm{span}\{\widehat{S}_0, \widehat{S}_1, \ldots, \widehat{S}_d\}$ is a Schur ring over $G$ with intersection numbers $q_{i,j}^k$. Since $\Gamma$ has a $Q$-polynomial ordering $\widehat{S}_{\tau(0)}, \widehat{S}_{\tau(1)}, \ldots, \widehat{S}_{\tau(d)}$, we claim that the (symmetric) association scheme derived from the Schur ring $\widehat{\mathcal{S}}(G)$   is  $P$-polynomial with an ordering of relations $\widehat{S}_{\tau(0)}, \widehat{S}_{\tau(1)}, \ldots, \widehat{S}_{\tau(d)}$. Therefore, by Lemma \ref{lem::P-polynomial}, the Cayley graph $\widehat{\Gamma} = \mathrm{Cay}(G, \widehat{S}_{\tau(1)})$ is a distance-regular graph of diameter $d$ with intersection 
		numbers $q_{i,j}^k$.
	\end{proof}
	
	The Cayley graph $\widehat{\Gamma}=\mathrm{Cay}(G, \widehat{S}_{\tau(1)})$ in Lemma \ref{lem::dual_graph} is called the \textit{dual graph} of the $Q$-polynomial distance-regular graph $\Gamma=\mathrm{Cay}(G, S)$.

	A graph is called \textit{integral} if all its eigenvalues are integers.   Let $\mathcal{F}_G$ be the set of all subgroups of $G$. The \textit{Boolean algebra}  $\mathbb{B}(\mathcal{F}_G)$ is the set whose elements are obtained by  arbitrary finite intersections, unions, and complements of the elements in $\mathcal{F}_G$. The minimal non-empty elements of $\mathbb{B}(\mathcal{F}_G)$ are called \textit{atoms}. It is known that each element of $\mathbb{B}(\mathcal{F}_G)$ is the union of some atoms, and the atoms for $\mathbb{B}(\mathcal{F}_G)$ are the sets $[g] = \{x\in G \mid \langle x\rangle= \langle g\rangle\}$, $g\in G$.  The following lemma provides a characterization for integral Cayley graphs over abelian groups.
	
	\begin{lemma}[\cite{ICG}]\label{lem::int_Cay}
		Let $G$ be an abelian group, and let $S$ be an inverse closed subset of $G$ with $e\notin S$. Then the Cayley graph $\mathrm{Cay}(G,S)$ is integral if and only if  $S\in \mathbb{B}(\mathcal{F}_G)$.
	\end{lemma}

	\section{Finite geometry}\label{section::4}	
	
	In this section, we introduce some classic results in finite geometry, which play a key role in the proof of our main result.

	Let $G$ be a finite group, and let  $N$ be a proper subgroup of $G$ with order $|N| = r$ and index $[G:N] = m$. A $k$-subset $D$ of $G$ is called an $(m, r, k, \mu)$-\textit{relative difference set} relative to $N$ (the forbidden subgroup) if and only if
	$$\underline{D} \cdot \underline{D}^{(-1)} = k \cdot e + \mu \cdot \underline{G \setminus N}.$$
	
	\begin{lemma}[{\cite[Theorem 4.1.1]{RDS}}]\label{lem::NRDS}
		Let $D$ be a $(nm,n,nm,m)$-relative difference set relative to $N$ in an abelian group $G$. Let $g$ be an element in $G$. Then the order of $g$ divides $nm$, or $n=2$, $m=1$ and $G\cong\mathbb{Z}_4$.
	\end{lemma}

	A subset $D$ of $G$ is called a \textit{polynomial addition set} if there exists a polynomial $f(x) \in \mathbb{Z}[x]$ with degree $\deg f \geq 1$ such that $f(\underline{D})=m\underline{G}$ for some integer $m$. In this context, we also describe $D$ as a $(v, k, f(x))$-polynomial addition set, where $|G|=v$ and $|D| = k$. If $G$ is cyclic, then $D$ is called a $(v, k, f(x))$-\textit{cyclic polynomial addition set}.
	
	\begin{lemma}[{\cite[Corollary 5.4.5]{MaPhD}}]\label{lem::PAS1}
		There is no $(v,k,x^n-b)$-cyclic polynomial addition set with $1<k<v-1$ and $n\geq 1$.
	\end{lemma}
	
	The proof of Lemma \ref{lem::PAS1} relies on the following crucial lemma from  \cite{MaPhD}, which is also useful in the proof of our main result.
	
	\begin{lemma}[{\cite[Lemma 1.5.1]{S02}}, {\cite[Lemma 3.2.3]{MaPhD}}]\label{lem::Ma}
		Let $p$ be a prime and let $G$ be an abelian group with a cyclic Sylow $p$-subgroup $S$. If $\underline{Y}\in \mathbb{Z} G$ satisfies $\chi(\underline{Y})\equiv0 ~(\mathrm{mod}~p^a)$ for all characters $\chi\in \widehat{G}$ of order divisible by $|S|$, then there exist $\underline{X_1},\underline{X_2}\in \mathbb{Z}G$ such that $\underline{Y}=p^a\underline{X_1}+\underline{P}\cdot\underline{X_2}$, where $P$ is the unique subgroup of order $p$ of $G$. Furthermore, if $\underline{Y}$ has non-negative coefficients only, then $X_1$ and $X_2$ also can be chosen to have non-negative coefficients only.
	\end{lemma}
	
	A \textit{transversal design} $TD(r,v)$ of order $v$ with line size $r$ ($r\leq v$) is a triple $(\mathcal{P},\mathcal{G},\mathcal{L})$ such that 
	\begin{enumerate}[(i)]
		\item $\mathcal{P}$ is a set of $rv$ elements (called \textit{points});
		\item $\mathcal{G}$ is a partition of $\mathcal{P}$ into $r$ classes, each of size $v$ (called \textit{groups}); 
		\item  $\mathcal{L}$ is a collection of $r$-subsets of $\mathcal{P}$ (called \textit{lines});
		\item two distinct points are contained in  a unique line if and only if they are in distinct groups.
	\end{enumerate}
	It follows immediately that $|G \cap L| = 1$ for every $G \in \mathcal{G}$ and every $L \in \mathcal{L}$, and $|\mathcal{L}| = v^2$. The \textit{line graph} of a transversal design $TD(r, v)$ is the graph with lines as vertices and two of them being adjacent whenever there is a point incident to both lines. It is known that the line graph of a transversal design $TD(r, v)$ is a strongly regular graph with parameters $(v^2, r(v-1), v + r^2 - 3r, r^2 - r)$ (cf. \cite[p. 122]{Jur95}), and so has exactly three distinct eigenvalues, namely $r(v-1)$, $v-r$, and $-r$. 
	For $r=2$ we obtain the lattice graph $K_v \times K_v$, and for $r=v$ we obtain the complete multipartite graph $K_{v \times v}$.
	
	The following lemma provides certain restrictions for the  antipodal cover of the line graph of a transversal design $TD(r, v)$ with $r\leq v$. 
	\begin{lemma}[{\cite[Proposition 2.4]{Jur98}}]\label{lem::antipodal cover of line graph}
		An antipodal cover of the line graph of a transversal design $TD(r,v)$, $r\leq v$, has diameter four when $r=2$ and diameter three otherwise. 
	\end{lemma}
	
	Let $p$ be an odd prime. In \cite{ZLH23}, it was shown that  every distance-regular Cayley graph over $\mathbb{Z}_{p}\oplus\mathbb{Z}_{p}$  is  the line graph of a transversal design.
	
	\begin{lemma}[{\cite[Lemma 3.1]{ZLH23}}]\label{lem::1}
		Let $p$ be an odd prime, and let $\Gamma=\mathrm{Cay}(\mathbb{Z}_{p}\oplus\mathbb{Z}_{p},S)$ be a Cayley graph over $\mathbb{Z}_{p}\oplus\mathbb{Z}_{p}$. Then $\Gamma$ is  distance-regular if and only if $S=\cup_{i=1}^rH_i\setminus\{(0,0)\}$, where $2\leq r\leq p+1$, and $H_i$ ($i=1,\ldots,r$) are  subgroups of order $p$ in $\mathbb{Z}_{p}\oplus\mathbb{Z}_{p}$. In this situation, $\Gamma$ is isomorphic to the line graph of a transversal design $TD(r,p)$ when $r\leq p$, and to a complete graph when $r=p+1$. In particular, $\Gamma$ is primitive if and only if $2\leq r\leq p-1$ or $r=p+1$, and $\Gamma$ is imprimitive if and only if $r=p$, in which case  $\Gamma$ is the complete multipartite graph $K_{p\times p}$.
	\end{lemma}

	A \textit{clique} in a graph 
	$\Gamma$
	is a subgraph in which every pair of vertices are adjacent. A \textit{maximal clique} is a clique that cannot be extended by including an additional vertex that is adjacent to all its vertices.  The \textit{clique number} of $\Gamma$ is the cardinality of a clique of maximum size in $\Gamma$.

	Let $\mathbb{F}_p$ denote the finite field of order $p$ with $p$ being an odd prime. It is known that the Desarguesian affine plane $AG(2, p)$ can be identified with $\mathbb{F}_p^2$. Let $U = \{(a_j, b_j) \mid 1 \leq j \leq \ell\}$ be an $\ell$-subset of $\mathbb{F}_p^2$. We define  \[\mathrm{Dir}(U) = \left\{ \frac{b_j - b_k}{a_j - a_k} \mid 1 \leq j \neq k \leq \ell \right\}.\] Then the elements of $\mathrm{Dir}(U)$ are called the \textit{directions} determined by $U$. Let $W$ be a subset of $AG(2, p)$ with $1 < |W| \leq p$. According to \cite[Theorem 5.2]{S99}, $W$ determines either a single direction (in this case, $W$ is called \textit{collinear}) or satisfies the inequality
	\begin{equation}\label{Direction}
		|\mathrm{Dir}(W)| \geq \frac{|W| + 3}{2}.
	\end{equation}
	Note that $\mathbb{Z}_p \oplus \mathbb{Z}_p$ coincides with $\mathbb{F}_p^2$ as sets. For any subset $B$ of $\mathbb{Z}_p \oplus \mathbb{Z}_p$ with $(0, 0)\in B$, we see that $B$ is collinear if and only if $B$ is contained in some subgroup of order $p$ in $\mathbb{Z}_p \oplus \mathbb{Z}_p$. Moreover, if $K_1$ and $K_2$ are two distinct subgroups of order $p$ in $\mathbb{Z}_p \oplus \mathbb{Z}_p$, then $\mathrm{Dir}(K_1) \neq \mathrm{Dir}(K_2)$.

	\begin{lemma}\label{lem::key0}
		Let $p$ be an odd prime. Suppose that $\Gamma=\mathrm{Cay}(\mathbb{Z}_{p}\oplus\mathbb{Z}_{p},S)$ with $S=\cup_{i=1}^rH_i\setminus\{(0,0)\}$, where $2\leq r\leq p-1$, and $H_i$ ($i=1,\ldots,r$) are subgroups of order $p$ in $\mathbb{Z}_{p}\oplus\mathbb{Z}_{p}$.
		Then the following statements hold.
		\begin{enumerate}[(i)]
			\item The clique number of $\Gamma$ is equal to $p$.
			\item If $C$ is a non-collinear clique in $\Gamma$, then 
			$|C|\leq 2r-3$.
			\item If $\Gamma$ is the halved graph of a bipartite distance-regular graph $\Gamma'$, then for $\Gamma'$, we have $\mu=1$.
		\end{enumerate}
	\end{lemma}
	\begin{proof}
		(i) Clearly, $H_i$ is a clique of order $p$ in  $\Gamma$. Since $\Gamma$ has eigenvalues $r(p-1)$, $p-r$ and $-r$, the Delsarte bound (cf. \cite[Section 3.3.2]{D73}) implies that the clique number of $\Gamma$ is at most $1 - \frac{r(p-1)}{-r} = p$. Therefore, the clique number of $\Gamma$ is exactly  $p$.
		
		(ii) Since every pair of vertices in $C$ are adjacent, we assert that  the directions determined by $C$ are contained in the set $\{\mathrm{Dir}(H_i)\mid 1\leq i\leq r\}$, and hence  $r \geq |\mathrm{Dir}(C)|$. As $C$ is non-collinear, from \eqref{Direction} we obtain $\mathrm{Dir}(C) \geq (|C| + 3)/2$. Therefore, $|C|\leq 2r-3$.
		
		(iii) Suppose that $\Gamma$ is the halved graph of  a bipartite distance-regular graph $\Gamma'$. Let $k$ and $\mu$ denote the valency of $\Gamma'$ and the number of common neighbors of two vertices at distance two in $\Gamma'$, respectively. By contradiction, assume that $\mu\geq 2$. By \cite[Proposition 4.2.2]{BCN89},  we have
		\begin{equation}\label{k2-k}
			\frac{k^2-k}{\mu}=r(p-1).
		\end{equation}
		For any $v\in V(\Gamma')$, let $S(v)$ be the neighborhood of $v$ in $\Gamma'$. By {\cite[Lemma 2]{H86}}, $S(v)$ is a maximal clique in $\Gamma'^+$ or $\Gamma'^-$. Clearly, the  maximal cliques  $S(v)$, for $v\in V(\Gamma')$, must cover $\Gamma'^+$ and $\Gamma'^-$. Thus  there exists some vertex $v\in V(\Gamma')$ such that $C:=S(v)$ is a maximal clique in  $\Gamma$ containing the vertex $(0,0)$. According to (i), the clique number of $\Gamma$ is $p$, and $k=|S(v)|=|C|\leq p$. Thus we have $r\mu \leq p$ by \eqref{k2-k}. If $C$ is collinear, then $k=|C|=p$ because $C$ is a maximal clique. By \eqref{k2-k}, we have $r\mu=p$, and so $r=1$ or $p$, contrary to our assumption. If $C$ is non-collinear, then (ii) indicates that $k=|C|\leq 2r-3<2r$. Combining this with \eqref{k2-k} and  $\mu\geq 2$, we obtain $2r>p$, which is impossible because $p\geq \mu r\geq 2r$.
	\end{proof}


	\section{Imprimitive distance-regular Cayley graphs with diameter three over abelian groups} \label{section::5}
	
	In this section, we present  some properties of imprimitive distance-regular Cayley graphs with diameter $3$ over abelian groups.

	It is known that an antipodal bipartite distance-regular graph with diameter $3$ is a complete bipartite graph without a perfect matching. Also, by \cite[Corollary 8.2.2]{BCN89}, every non-antipodal bipartite distance-regular graph with diameter $3$ is $Q$-polynomial. Therefore, by Lemma \ref{lem::dual imprimitivity} and Lemma \ref{lem::dual_graph}, we can deduce the following result  immediately.
	
	\begin{prop}\label{prop::non-antipodal_bipartite}
		Let $G$ be an abelian group. If $\Gamma$ is a non-antipodal bipartite distance-regular Cayley graph with diameter $3$ over $G$, then its dual graph $\widehat{\Gamma}$   is an antipodal non-bipartite distance-regular Cayley graph with diameter $3$ over  $G$.
	\end{prop}
	
	By Proposition \ref{prop::non-antipodal_bipartite} and the above arguments,  in order to study distance-regular Cayley graphs with diameter $3$ over abelian groups, the primary task is to consider those that are antipodal and non-bipartite.

	For the sake of convenience, we maintain the following notation throughout the remainder of this paper.

	{\flushleft \textbf{Notation.}} Let $G$ and $H$ be finite abelian groups under addition, and let  $G \oplus H$ denote the direct product of $G$ and $H$. For subsets $A \subseteq G$, $B \subseteq H$, and elements $g \in G$, $h \in H$, we define $g + A = \{g + a \mid a \in A\}$, $(g, B) = \{(g, b) \mid b \in B\}$,
	$(A, h) = \{(a, h) \mid a \in A\}$,
	and $(A, B) = \{(a, b) \mid a \in A, b \in B\}$.
	If $\Gamma=\mathrm{Cay}(G \oplus H,S)$ is a Cayley graph over $G \oplus H$, then the connection set $S$ can be  expressed as
	$$S = \cup_{h \in H} (R_h, h)=\cup_{g \in G} (g, L_g),$$
	where $R_h$  is a subset of $G$ such that $0_G \not\in R_{0_H}$ 
	and $R_h = -R_{-h}$ for all $h \in H$, and $L_g$ is a subset of $H$ such that $0_H\notin L_{0_G}$ and $L_g = -L_{-g}$ for all $g \in G$. Let 
	$R = \cup_{h\in H}R_h$ and $L = \cup_{g\in G}L_g$.
	Furthermore, if $\Gamma$ is distance-regular, then 
	we denote by $k$, $\lambda$, $\mu$ and $d$ the valency, the number of common neighbors of two adjacent vertices, the number of common neighbors of two vertices at distance $2$, and the diameter of $\Gamma$, respectively.
	
	Note that the set of irreducible characters of $G \oplus H$ can be represented as $\widehat{G \oplus H}=\{(\chi, \psi) \mid \chi \in \widehat{G}, \psi \in \widehat{H}\}$, where the pair $(\chi, \psi)$ is defined such that $(\chi, \psi)((g, h)) = \chi(g)\psi(h)$ for every $(g, h) \in G \oplus H$. 
	
	\begin{lemma}\label{lem::character_sum}
		Let $G$ and $H$ be finite abelian groups under addition, and let $S=\cup_{h\in H}(R_h,h)$ $=\cup_{g\in G}(g,L_g)$ be a subset of  $G\oplus H$, where $R_h\subseteq G$ for all $h\in H$ and $L_g\subseteq H$ for all $g\in G$. Then
		\begin{equation*}
			\chi(\underline{R_{0_H}})=\frac{1}{|H|}\sum_{\psi\in \widehat{H}}(\chi,\psi)(\underline{S})
		\end{equation*}
		for every $\chi \in \widehat{G}$, and 
		\begin{equation*}
			\psi(\underline{L_{0_G}})=\frac{1}{|G|}\sum_{\chi\in \widehat{G}}(\chi,\psi)(\underline{S})
		\end{equation*}
		for every $\psi\in \widehat{H}$.
	\end{lemma}
	
	\begin{proof}
		By symmetry, we only need to prove the first part of the lemma. Note that 
		\begin{equation*}
			\sum_{\psi \in \widehat{H}} \psi(h) = 
			\begin{cases} 
				|H|, & \mbox{if}~h = 0_H; \\
				0, & \mbox{otherwise}.
			\end{cases}
		\end{equation*}
		For every $\chi \in \widehat{G}$, we have
		\begin{equation*}
			\begin{aligned}
				\sum_{\psi\in \widehat{H}}(\chi,\psi)(\underline{S})=&\sum_{\psi\in \widehat{H}}(\chi,\psi)\left(\sum_{h\in H}\underline{(R_h,h)}\right)
				=\sum_{\psi\in \widehat{H}}\left(\sum_{h\in H}\chi(\underline{R_h})\psi(h)\right)\\
				=&\sum_{h\in H}\left(\sum_{\psi\in \widehat{H}}\psi(h)\right)\chi(\underline{R_h})
				=|H|\cdot\chi(\underline{R_{0_H}}),
			\end{aligned}
		\end{equation*}
		and the result follows.
	\end{proof}
	
	\begin{prop}\label{prop::prime_r_antipodal}
		Let $G$ be an abelian group, and let $\Gamma$ be an $r$-fold antipodal non-bipartite distance-regular Cayley graph with diameter $3$ over  $G$. If $r$ is prime, then $G\cong M \oplus \mathbb{Z}_r$ and $\Gamma$ is isomorphic to a Cayley graph over $M \oplus \mathbb{Z}_r$ in which the antipodal class containing the identity vertex is $S_0 \cup S_3 = (0_M,\mathbb{Z}_r)$, where $M$ is an abelian group of order $|G|/r$.
	\end{prop}
	\begin{proof}
		Since $r$ is a prime divisor of $|G|$, we can  express $G$ as
		\begin{equation*}
			G = K\oplus \mathbb{Z}_{r^{s_1}} \oplus \mathbb{Z}_{r^{s_2}} \oplus \cdots \oplus \mathbb{Z}_{r^{s_t}},
		\end{equation*}
		where $s_1\geq s_2\geq\cdots\geq s_t$ and $r\nmid |K|$.  Let $H$ denote the antipodal class of $G$ containing the identity vertex, that is, $H=S_0\cup S_3$.  Since $|H|=r$ is prime and $K$ contains no elements of order $r$, we claim that every element of $S_3$ is of the form $(0_K,a_1,\ldots,a_t)$ with  $a_i\in \mathbb{Z}_{r^{s_i}}$ and $r^{s_i-1}\mid a_i$ for $1\leq i\leq t$. Thus there exists some $l\in \{1,\ldots,t\}$ such that $b=(0_K,b_1,\ldots,b_l=r^{s_l-1},0,\ldots,0)\in S_3$ with $r^{s_i-1}\mid b_i$ for $1\leq i\leq l$. Then  $H=\langle b\rangle$ because $|H|=r$ is prime. Let $\sigma$ be the mapping on $G$ defined by letting \[\sigma((k,i_1,\ldots,i_t)) = \left(k,i_1-\frac{b_1}{r^{s_l-1}}i_l,\ldots,i_{l-1}-\frac{b_{l-1}}{r^{s_l-1}}i_l,i_l,\ldots,i_t\right)\] for all $(k,i_1,\ldots,i_t)\in G=K\oplus \mathbb{Z}_{r^{s_1}} \oplus \mathbb{Z}_{r^{s_2}} \oplus \cdots \oplus \mathbb{Z}_{r^{s_t}}$. Clearly,  $\sigma$ is an automorphism of $G$, and $\sigma(b) = (0_K,0,\ldots,0,r^{s_l-1},0,\ldots,0)$. Then  $G\cong M\oplus \mathbb{Z}_m$, where $m=r^{s_l}$ and $M=K\oplus(\oplus_{i=1,i\neq l}^t\mathbb{Z}_{r^{s_i}})$, and $\Gamma$ is isomorphic to a Cayley graph over $M\oplus \mathbb{Z}_m$ in which the  antipodal class containing the identity vertex is $S_0\cup S_3= (0_M,\frac{m}{r}\mathbb{Z}_m)$. Therefore, it suffices to prove  $m=r$. We consider the following three cases. 
		
		{\flushleft \bf Case A.} $r=2$.
		
		In this situation, $S_3=\{(0_M,\frac{m}{2})\}$. Assume that $4\mid m$. If $(0_M,\frac{m}{4})\in S_1$, then $(0_M,-\frac{m}{4})=(0_M,\frac{m}{2})+(0_M,\frac{m}{4})\in S_2$, which is impossible due to $(0_M,-\frac{m}{4})\in -S_1=S_1$.  If $(0_M,\frac{m}{4})\in S_2$, then $(0_M,\frac{m}{4})$ and $(0_M,\frac{m}{2})$ are adjacent, and hence $(0_M,\frac{m}{4})=(0_M,\frac{m}{2})-(0_M,\frac{m}{4})\in S_1$, a contradiction. Thus we have $(0_M,\frac{m}{4})\notin S_1\cup S_2\cup S_3$, which is impossible. Therefore, we conclude that $m=2=r$, as desired.
		
		{\flushleft \bf Case B.} $r\neq2$ and $\lambda=\mu$.
		
		By Lemma \ref{lem::antipodal_DRG}, we have $\frac{m|M|}{r}-1=k=1+r\mu$, and hence $ \frac{m|M|}{r}\equiv2 ~(\mathrm{mod}~r)$.
		As $r\neq 2$, we assert that $m=r$, as required.
		
		{\flushleft \bf Case C.} $r\neq2$ and $\lambda\neq\mu$.
		
		In this case, by Lemma \ref{lem::antipodal_DRG}, $\Gamma$ is integral, and so $(\chi,\psi)(\underline{S})\in \mathbb{Z}$ for all $(\chi,\psi)\in \widehat{M\oplus \mathbb{Z}_m}$.
		By Lemma \ref{lem::character_sum}, for any $\psi\in \widehat{\mathbb{Z}_m}$, we have
		$$
		\psi(\underline{L_{0_M}})=\frac{1}{|M|}\sum_{\chi\in \widehat{M}}(\chi,\psi)(\underline{S})\in \mathbb{Q},
		$$
		and hence  $\psi(\underline{L_{0_M}})\in \mathbb{Z}$ because it is an algebraic integer. Note that $\{\psi(\underline{L_{0_M}})\mid \psi\in \widehat{\mathbb{Z}_m}\}$ gives a complete set of eigenvalues of the Cayley graph $\mathrm{Cay}(\mathbb{Z}_m,L_{0_M})$. Hence, by Lemma \ref{lem::int_Cay}, $L_{0_M}$ is a union of some atoms for $\mathbb{B}(\mathcal{F}_{\mathbb{Z}_m})$. On the other hand, by Corollary \ref{cor::DR}, 
		\begin{equation*}
			\underline{(0_M,\frac{m}{r}\mathbb{Z}_m)}\cdot\underline{S\cup\{(0_M,0)\}}=\underline{M\oplus\mathbb{Z}_m},
		\end{equation*}
		and it follows that
		\begin{equation*}
			\underline{(0_M,L_{0_M}\cup\{0\})}\cdot \underline{(0_M,\frac{m}{r}\mathbb{Z}_m)}=\underline{(0_M,\mathbb{Z}_m)},
		\end{equation*}
		or equivalently,
		\begin{equation*}
			\underline{L_{0_M}\cup\{0\}}\cdot \underline{\frac{m}{r}\mathbb{Z}_m}=\underline{\mathbb{Z}_m}.
		\end{equation*}
		This implies that  $L_{0_M}\cup\{0\}$ contains exactly one element from each  coset of $\frac{m}{r}\mathbb{Z}_m$ in $\mathbb{Z}_m$. Then there exists an element $a\in L_{0_M}$ such that  $(1+\frac{m}{r}\mathbb{Z}_m)\cap L_{0_M}=\{a\}$. If  $1+\frac{m}{r}\mathbb{Z}_m\subseteq\mathbb{Z}_m^*$, then $a\in\mathbb{Z}_m^*$. Since $L_{0_M}$ is a union of some atoms for $\mathbb{B}(\mathcal{F}_{\mathbb{Z}_m})$ and $\mathbb{Z}_m^*$ is exactly an atom, we assert that  $1+\frac{m}{r}\mathbb{Z}_m\subseteq\mathbb{Z}_m^*\subseteq L_{0_M}$. Thus $1+\frac{m}{r}\mathbb{Z}_m=\{a\}$, and it follows that $m=r$, as desired. If $1+\frac{m}{r}\mathbb{Z}_m\not\subseteq\mathbb{Z}_m^*$, then there exists some  $i\in \mathbb{Z}_m$ such that $1+i\frac{m}{r}\notin \mathbb{Z}_m^*$. Combining this with $\mathrm{gcd}(1+i\frac{m}{r},\frac{m}{r})=1$, we obtain $\mathrm{gcd}(1+i\frac{m}{r},m)=r$, which gives that $m=r$ because $m$ is a power of $r$. The result follows.	
	\end{proof}
	
	\begin{prop}\label{prop::an1}
		
		Let $G$  be an abelian group, and let $p$ be a prime. Assume that $\Gamma = \mathrm{Cay}(G \oplus \mathbb{Z}_p, S)$ is an antipodal non-bipartite distance-regular Cayley graph with diameter $3$ in which the antipodal class containing the identity vertex is $S_0\cup S_3 = (0_G, \mathbb{Z}_p)$. Let $k\geq \theta_1\geq -1 \geq \theta_3$ be all distinct eigenvalues of $\Gamma$ and $2\delta=\theta_1-\theta_3$.
		Then the following statements hold.
		\begin{enumerate}[(i)]
			\item The sets $R_i$, for $i \in \mathbb{Z}_p$, form a partition of $G \setminus \{0_G\}$.
			\item If $p>2$, then $\frac{|G|}{2\delta}$, $\theta_1$ and $-\theta_3$ are positive integers. Moreover, for every non-trivial character $\psi \in \widehat{\mathbb{Z}_p}$, the set $B=\{g\in G\mid (\chi_g,\psi)(\underline{S})=\theta_1\}$ is a $(|G|, -\frac{|G|}{2\delta}\theta_3, x^p - \frac{|G|^p}{(2\delta)^p})$-polynomial addition set such that
			\begin{equation*}
				\chi_l(\underline{B}) = \frac{|G|}{2\delta}\left( \sum_{i \in \mathbb{Z}_p} \psi(i) a_{-l}(\underline{R_i}) - \theta_3 \cdot a_{-l}(0_G) \right)~\mbox{for all $l\in G$}.
			\end{equation*}
			\item If $p=2$, then there exists a strongly regular Cayley graph over $G$ with parameters $(|G|,\frac{|G|}{2\delta}\theta, \frac{|G|}{4\delta^2}(\theta^2-1),\frac{|G|}{4\delta^2}(\theta^2-1))$, where $\theta=\theta_1$ or $-\theta_3$.  
		\end{enumerate}
	\end{prop}
	\begin{proof}
		(i)	By Corollary \ref{cor::DR}, we have
		\begin{equation*}
			\underline{G\oplus \mathbb{Z}_p}=\underline{(0_G,\mathbb{Z}_p)}\cdot\left(\sum_{i\in \mathbb{Z}_p}\underline{(R_{i},i)}+e\right)=\sum_{i\in \mathbb{Z}_p}\underline{(R_{i},\mathbb{Z}_p)}+\underline{(0_G,\mathbb{Z}_p)}.
		\end{equation*}
		Therefore, the sets $R_i$ ($i \in \mathbb{Z}_p$) form a partition of $G \setminus \{0_G\}$. 
		
		(ii) Again by Corollary \ref{cor::DR},   
		\begin{equation}\label{S}
			\underline{S}^2=k\cdot e+(\lambda-\mu) \underline{S}+\mu (\underline{G\oplus \mathbb{Z}_p}-\underline{(0_G,\mathbb{Z}_p)}).
		\end{equation} 
		Let $\psi\in \widehat{\mathbb{Z}_p}$ be a non-trivial character of $\mathbb{Z}_p$. We have $\psi(\underline{\mathbb{Z}_p})=0$. For any $g\in G$, let $\chi_g\in \widehat{G}$ be the character of $G$ defined in \eqref{equ::char_Ab}. Then $(\chi_g,\psi)$ is a non-trivial character of $\widehat{G\oplus \mathbb{Z}_p}$, and so $(\chi_g,\psi)(\underline{\widehat{G\oplus \mathbb{Z}_p}})=\chi_g(\underline{G})\cdot \psi(\underline{\mathbb{Z}_p})=0$.  By applying $(\chi_g,\psi)$ on both sides of \eqref{S}, we obtain 
		\begin{equation*}
			((\chi_g,\psi)(\underline{S}))^2=k+(\lambda-\mu) (\chi_g,\psi)(\underline{S}),
		\end{equation*}
		which implies that    $(\chi_g,\psi)(\underline{S})=\theta_1$ or $\theta_3$ for all $g\in G$ according to Lemma \ref{lem::antipodal_DRG}. Then   
		\begin{equation*}
			\sum_{g\in G}(\chi_g,\psi)(\underline{S})\cdot g	=\theta_1\underline{B}+\theta_3\underline{G\setminus B}
			=2\delta\underline{B}+\theta_3\underline{G}.
		\end{equation*}
		Since $\underline{S}=\sum_{i\in \mathbb{Z}_p}\underline{(R_i,i)}$, we have 
		\begin{equation}\label{S1}	2\delta\underline{B}+\theta_3\underline{G}=\sum_{g\in G}\left(\sum_{i\in \mathbb{Z}_p}\psi(i)\chi_{g}(\underline{R_i})\right)g.
		\end{equation}
		Let $l\in G$. By applying the character $\chi_{l}\in \widehat{G}$ on both sides of \eqref{S1}, we get 	\begin{equation}\label{S2}
			\begin{aligned}
				2\delta\chi_{l}(\underline{B})+\theta_3\chi_{l}(\underline{G})=&\sum_{g\in G}\left(\sum_{i\in \mathbb{Z}_p}\psi(i)\chi_{g}(\underline{R_i})\right)\chi_{l}(g)=\sum_{g\in G}\left(\sum_{i\in \mathbb{Z}_p}\psi(i)\chi_{g}(\underline{R_i})\right)\chi_{g}(l)\\
				=&\sum_{i\in \mathbb{Z}_p}\psi(i) \left(\sum_{g\in G}\chi_{g}(\underline{R_i})\chi_{g}(l)\right)=\sum_{i\in \mathbb{Z}_p}\psi(i)\cdot |G|a_{-l}(\underline{R_i}),
			\end{aligned}
		\end{equation} 
		where the last equality follows from the Fourier inversion formula \eqref{equ::FI}.
		Note that $\chi_l(\underline{G})=|G|$ if $l=0_G$, and $\chi_l(\underline{G})=0$ otherwise. By  \eqref{S2}, we obtain   	
		\begin{equation}
			\label{c1}
			\chi_l(\underline{B}) = \frac{|G|}{2\delta}\left( \sum_{i \in \mathbb{Z}_p} \psi(i) a_{-l}(\underline{R_i}) - \theta_3 \cdot a_{-l}(0_G) \right),
		\end{equation}	
		and it follows that $|B|=-\frac{|G|}{2\delta}\theta_3>0$. Recall that the sets $R_i$, for $i\in \mathbb{Z}_p$, form a partition of $G\setminus\{0_G\}$. Then from \eqref{c1} we can deduce that
		\begin{equation*}
			\begin{aligned}
				\chi_{l}(\underline{B}^p)&=\chi_{l}(\underline{B})^p
				=\frac{|G|^p}{(2\delta)^p}\left(\sum_{i\in \mathbb{Z}_p}a_{-l}(\underline{R_i})+(-\theta_3)^p\cdot a_{-l}(0_G)\right)\\
				&=\frac{|G|^p}{(2\delta)^p}\left(1+((-\theta_3)^p-1)\cdot a_{-l}(0_G)\right).
			\end{aligned}	
		\end{equation*}
		Using the Fourier inversion formula \eqref{equ::FI}, we get 
		\begin{equation}\label{equ::Bm}
			\underline{B}^p=\frac{|G|^{p-1}}{(2\delta)^p}(((-\theta_3)^p-1)\cdot\underline{G}+|G|\cdot 0_G).
		\end{equation}
		By Lemma \ref{lem::antipodal_DRG}, $(2\delta)^2=4k+(\lambda-\mu)^2\in \mathbb{Z}$. As $p$ is odd and $\underline{B}^p\in \mathbb{Z}G$, from \eqref{equ::Bm} and $(2\delta)^2\in \mathbb{Z}$ we can deduce that $\frac{|G|}{2\delta}$ and $2\delta$ are integers, and so are $\theta_1$ and $\theta_3$. Moreover, again by  \eqref{equ::Bm}, we assert that $B$ is a $(|G|,-\frac{|G|}{2\delta}\theta_3,x^p-\frac{|G|^p}{(2\delta)^p})$-polynomial addition set in $G$. 
		
		(iii) If $p=2$,  then $\mathbb{Z}_2$ has only one non-trivial character, namely $\psi_1$, where $\psi _1(0)=1$ and $\psi _1(1)=-1$. By substituting $\psi=\psi_1$ in \eqref{c1}, we obtain
		\begin{equation*}
			\chi_l(\underline{B})=\frac{|G|}{2\delta}\left(a_{-l}(\underline{R_0})-a_{-l}(\underline{R_1}) - \theta_3 \cdot a_{-l}(0_G) \right) ~\mbox{for  all $l\in G$},
		\end{equation*} 
		which implies $-B=B$. Moreover, by \eqref{equ::Bm}, we have
		\begin{equation}\label{equ::Str_Reg_1}
			\underline{B}^2=\frac{|G|}{(2\delta)^2}((\theta_3^2-1)\cdot\underline{G}+|G|\cdot 0_G).
		\end{equation}
		If $0_G\notin B$, then \eqref{equ::Str_Reg_1} indicates that the Cayley graph $\mathrm{Cay}(G,B)$ is with diameter $d\leq 2$.
		If $d=1$, then $B=G\setminus\{0_G\}$, and hence $|G|-1=|B|=-\frac{|G|}{2\delta}\theta_3$. On the other hand, by Lemma \ref{lem::antipodal_DRG},  we have $\theta_3=\frac{\lambda-\mu}{2}-\delta$, $\delta=\sqrt{k+(\frac{\lambda-\mu}{2})^2}$, $|G|=k+1$ and $k=\mu+\lambda+1$. Thus, we obtain   $k-1=\lambda-\mu$ or $\mu-\lambda$, implying that  $\mu=0$ or $\lambda=0$, which is a contradiction. Therefore, $d=2$. Again by \eqref{equ::Str_Reg_1}, we assert that $\mathrm{Cay}(G,B)$ is a strongly regular Cayley graph with parameters $(|G|,-\frac{|G|}{2\delta}\theta_3, \frac{|G|}{4\delta^2}(\theta_3^2-1),\frac{|G|}{4\delta^2}(\theta_3^2-1))$. If $0_G\in B$, then $0_G\notin G\setminus B$. Combining \eqref{equ::Str_Reg_1} with $\theta_1-\theta_3=2\delta$ and $|B|=-\frac{|G|}{2\delta}\theta_3$ yields that
		\begin{equation}\label{equ::Str_Reg_2}
			(\underline{G\setminus B})^2=\frac{|G|}{(2\delta)^2}((\theta_1^2-1)\cdot\underline{G}+|G|\cdot 0_G).
		\end{equation}
		By a similar analysis,  we can deduce from \eqref{equ::Str_Reg_2} that $\mathrm{Cay}(G,G\setminus B)$ is a strongly regular Cayley graph with parameters $(|G|,\frac{|G|}{2\delta}\theta_1, \frac{|G|}{4\delta^2}(\theta_1^2-1),\frac{|G|}{4\delta^2}(\theta_1^2-1))$. 
	\end{proof}

	\section{Imprimitive distance-regular Cayley graphs with diameter four over abelian groups} \label{section::6}
	
	In this section, we present some properties of antipodal bipartite distance-regular Cayley graphs with diameter $4$ over abelian groups.
	
	
	\begin{prop}\label{prop::anti_bipart}
		Let $G$ be an abelian group, and let $\Gamma$ be an $r$-fold antipodal bipartite distance-regular Cayley graph with diameter $4$ over $G$. If $r$ is an odd prime, then $G\cong M\oplus \mathbb{Z}_r$ and $\Gamma$ is isomorphic to a Cayley graph over $M\oplus \mathbb{Z}_r$ in which  the antipodal class and the bipartition set containing the identity vertex are $S_0\cup S_4=(0_{M}, \mathbb{Z}_r)$ and $S_0\cup S_2 \cup S_4=M_1\oplus \mathbb{Z}_r$, resepctively, where $M$ is an abelian group of order $|G|/r$ and $M_1$ is an index $2$ subgroup of $M$.
	\end{prop}
	\begin{proof}
		Since $r$ is a prime divisor of $|G|$, as in Proposition \ref{prop::prime_r_antipodal}, we assert that $G\cong M\oplus \mathbb{Z}_m$ with $M$ being an abelian group of order $|G|/m$ and $m=r^{\ell}$ for some $\ell\geq 1$, and that $\Gamma$ is isomorphic to a Cayley graph $\Gamma'$ over $M\oplus \mathbb{Z}_m$ in which the antipodal class containing the identity vertex is $S_0\cup S_4=(0_M,\frac{m}{r}\mathbb{Z}_m)$.  Furthermore, since $\Gamma'$ is bipartite and $m=r^\ell$ is odd, we have $2\mid |M|$. Let $H$ be the bipartition set $S_0\cup S_2\cup S_4$ in $\Gamma'$. Then $H$ is an index $2$ subgroup of $M\oplus \mathbb{Z}_m$, and so
		$H=M_1 \oplus \mathbb{Z}_m$, where $M_1$ is an index $2$ subgroup of $M$. Therefore, it remains to prove  $m=r$. 
		
		Since $M$ is an abelian group of even order, we can assume that 
		$M=K\oplus\mathbb{Z}_{2^{s_1}} \oplus \mathbb{Z}_{2^{s_2}} \oplus \cdots \oplus \mathbb{Z}_{2^{s_t}}$,
		where $s_1\geq s_2\geq\cdots\geq s_t\geq 1$ ($t\geq 1$) and $2\nmid |K|$. Then  $M_1=K\oplus\mathbb{Z}_{2^{s_1}}  \oplus \cdots\oplus 2\mathbb{Z}_{2^{s_i}}\oplus\cdots \oplus \mathbb{Z}_{2^{s_t}}$ for some $i\in\{1,\ldots,t\}$. Let $m'=2^{s_i}m$. As $r$ is odd and $m=r^\ell$, we have $\mathbb{Z}_{2^{s_i}}\oplus \mathbb{Z}_{m}\cong \mathbb{Z}_{m'}$. Let $M'=K\oplus(\oplus_{j=1,j\neq i}^t \mathbb{Z}_{2^{s_i}})$. Then  $M\oplus \mathbb{Z}_m\cong M'\oplus \mathbb{Z}_{m'}$, and it is easy to check that $\Gamma'$ is isomorphic to a Cayley graph over $M'\oplus \mathbb{Z}_{m'}$ in which the antipodal class and the bipartition set containing the identity vertex are $S_0\cup S_4=(0_{M'}, \frac{m'}{r}\mathbb{Z}_{m'})$ and $S_0\cup S_2 \cup S_4=M'\oplus 2\mathbb{Z}_{m'}$, respectively. By Lemma \ref{lem::DR},  
		\begin{equation*}
			\underline{\left(0_{M'},\frac{m'}{r}\mathbb{Z}_{m'}\right)}\cdot\underline{S_1}=(\underline{S_0}+\underline{S_4})\cdot\underline{S_1}=\underline{S_1}+\underline{S_3}=\underline{(M', 1+2\mathbb{Z}_{m'})},
		\end{equation*}
		which implies that $\underline{\frac{m'}{r}\mathbb{Z}_{m'}}\cdot \underline{L_{0_{M'}}}=\underline{1+2\mathbb{Z}_{m'}}$. Thus $|L_{0_{M'}}|=\frac{m'}{2r}$, and $L_{0_{M'}}$ contains exactly one element from each coset in the set  $\{i+\frac{m'}{r}\mathbb{Z}_{m'}\mid i\in 1+2\mathbb{Z}_{m'}\}$.	Let $\widehat{\mathbb{Z}_{m'}}=\{\psi_g\mid g\in \mathbb{Z}_{m'}\}$ be the set of all irreducible characters of $\mathbb{Z}_{m'}$. In what follows, we shall determine the value of  $\psi_g(\underline{L_{0_{M'}}})$ for all $g\in \mathbb{Z}_{m'}$. Clearly,  $\psi_g(\underline{L_{0_{M'}}})=|L_{0_{M'}}|=\frac{m'}{2r}$ if $g=0$, and $\psi_g(\underline{L_{0_{M'}}})=-|L_{0_{M'}}|=-\frac{m'}{2r}$ if $g=\frac{m'}{2}$ due to $L_{0_{M'}}\subseteq 1+2\mathbb{Z}_{m'}$.
		Again by Lemma \ref{lem::DR}, we have 
		\begin{equation}\label{equ::S^2}
			\underline{S}^2=k\cdot e+\mu\underline{S_2}=k\cdot e+\mu\left(\underline{M'\oplus2\mathbb{Z}_{m'}}-\underline{\left(0_{M'},\frac{m'}{r}\mathbb{Z}_{m'}\right)}\right).
		\end{equation}
		If $g\in r\mathbb{Z}_{m'}\setminus \frac{m'}{2}\mathbb{Z}_{m'}$, then from \eqref{equ::S^2} and Lemma \ref{lem::antipodal_DRG} (ii) we obtain
		$(\chi,\psi_g)(\underline{S})^2=k-\mu r=0$, 
		and hence  $(\chi,\psi_g)(\underline{S})=0$ for all $\chi\in \widehat{M'}$.
		Therefore, by  Lemma \ref{lem::character_sum},
		\begin{equation*}
			\psi_g(\underline{L_{0_{M'}}})=\frac{1}{|M'|}\sum_{\chi\in \widehat{M'}}(\chi,\psi_g)(\underline{S})=0.
		\end{equation*}
		If $g\notin r\mathbb{Z}_{m'}$, then from  \eqref{equ::S^2} we get
		$((\chi,\psi_g)(\underline{S}))^2=k$, 
		and hence $(\psi,\chi_g)(\underline{S})\in \{-\sqrt{k},\sqrt{k}\}$ for all $\chi\in \widehat{M'}$.
		Again by Lemma \ref{lem::character_sum}, 
		\begin{equation}\label{Galois1}
			\psi_g(\underline{L_{0_{M'}}})=\frac{1}{|M'|}\sum_{\chi\in \widehat{M'}}(\chi,\psi_g)(\underline{S})\in\left\{\pm \frac{i\sqrt{k}}{|M'|}\mid i=0,1,\ldots,|M'|\right\}.     	
		\end{equation}
		We consider the following two cases.
		
		{\flushleft \bf Case A.} $\sqrt{k}\in \mathbb{Q}$.
		
		In this situation, for any $g\in \mathbb{Z}_{m'}$,  $\psi_g(\underline{L_{0_{M'}}})\in \mathbb{Z}$ because it is a rational algebraic integer. Thus $L_{0_{M'}}$ is a union of some atoms for $\mathbb{B}(\mathcal{F}_{\mathbb{Z}_{m'}})$. Furthermore, since $L_{0_{M'}}$ is a subset of $1+2\mathbb{Z}_{m'}$ that contains exactly one element from each coset in the set $\{i+\frac{m'}{r}\mathbb{Z}_{m'}\mid i\in 1+2\mathbb{Z}_{m'}\}$,  there exists some $a\in L_{0_{M'}}$ such that  $(1+\frac{m'}{r}\mathbb{Z}_{m'})\cap L_{0_{M'}}=\{a\}$. If  $1+\frac{m'}{r}\mathbb{Z}_{m'}\subseteq\mathbb{Z}_{m'}^*$, then $a\in\mathbb{Z}_{m'}^*$. As  $\mathbb{Z}_{m'}^*$ is exactly an atom, we assert that  $1+\frac{m'}{r}\mathbb{Z}_{m'}\subseteq\mathbb{Z}_{m'}^*\subseteq L_{0_{M'}}$. Thus $1+\frac{m'}{r}\mathbb{Z}_{m'}=\{a\}$, and it follows that $m'=r$, which is impossible. If $1+\frac{m'}{r}\mathbb{Z}_{m'}\not\subseteq\mathbb{Z}_{m'}^*$, then there exists some  $i\in \mathbb{Z}_{m'}$ such that $1+i\frac{m'}{r}\notin \mathbb{Z}_{m'}^*$. Combining this with $m'=2^{s_i}m=2^{s_i}r^\ell$ and $\mathrm{gcd}(1+i\frac{m'}{r},\frac{m'}{r})=1$, we obtain $\mathrm{gcd}(1+i\frac{m'}{r},m')=r$, which implies that   $r^2\nmid m'$. Therefore, we have $m=r$, and the result follows.		
		
		{\flushleft \bf Case B.} $\sqrt{k}\notin \mathbb{Q}$.
		
		In this situation, $\psi_g(\underline{L_{0_{M'}}})\in \mathbb{Q}$ for any $g\in r\mathbb{Z}_{m'}$, and $\psi_g(\underline{L_{0_{M'}}})\in \{\pm \frac{i\sqrt{k}}{|M'|}\mid i=0,1,\ldots,|M'|\}\subseteq\mathbb{Q}(\omega)\setminus \mathbb{Q}$ for any $g\notin r\mathbb{Z}_{m'}$, where $\omega$ is a primitive $m'$-th root of unity. Then there exists an element $\sigma_{c_0}$ with $c_0\in\mathbb{Z}_{m'}^*$ in the Galois group $\mathrm{Gal}[\mathbb{Q}(\omega):\mathbb{Q}]=\{\sigma_c:\omega\mapsto \omega^c\mid c\in\mathbb{Z}_{m'}^*\}$ such that $\sigma_{c_0}(\sqrt{k})=-\sqrt{k}$. By applying $\sigma_{c_0}$ on both sides of \eqref{Galois1}, we obtain 
		\begin{equation*}
			\psi_g(\underline{c_0L_{0_{M'}}})=-\frac{1}{|M'|}\sum_{\chi\in \widehat{M'}}(\chi,\psi_g)(\underline{S})=-\psi_g(\underline{L_{0_{M'}}})~\mbox{for $g\notin r\mathbb{Z}_{m'}$}.
		\end{equation*}
		Recall that $\psi_g(\underline{L_{0_{M'}}})=|L_{0_{M'}}|=\frac{m'}{2r}$ if $g=0$,  $\psi_g(\underline{L_{0_{M'}}})=-|L_{0_{M'}}|=-\frac{m'}{2r}$ if $g=\frac{m'}{2}$, and $\psi_g(\underline{L_{0_{M'}}})=0$ if $g\in r\mathbb{Z}_{m'}\setminus \frac{m'}{2}\mathbb{Z}_{m'}$. Thus  $\psi_g(\underline{c_0L_{0_{M'}}})=\psi_g(\underline{L_{0_{M'}}})\in\mathbb{Q}$ for all $g\in r\mathbb{Z}_{m'}$. According to the Fourier inversion formula \eqref{equ::FI}, for each $g\in\mathbb{Z}_{m'}$, we have  
		\begin{equation*}
			\begin{aligned}
				a_g(\underline{L_{0_{M'}}})+a_g(\underline{c_0L_{0_{M'}}})=&\frac{1}{m'}\sum_{h\in\mathbb{Z}_{m'}}\left(\psi_h(\underline{L_{0_{M'}}})+\psi_h(\underline{c_0L_{0_{M'}}})\right)\psi_{h}(g^{-1})\\
				=&\frac{1}{m'}\sum_{h\in\frac{m'}{2}\mathbb{Z}_{m'}}\left(\psi_h(\underline{L_{0_{M'}}})+\psi_h(\underline{c_0L_{0_{M'}}})\right)\psi_{h}(g^{-1})\\
				=&\frac{2}{m'}\sum_{h\in\frac{m'}{2}\mathbb{Z}_{m'}}\psi_h(\underline{L_{0_{M'}}})\psi_{h}(g^{-1}).
			\end{aligned}
		\end{equation*}
		Therefore, 
		\begin{equation*}
			\max_{g\in\mathbb{Z}_{m'}}\left(a_g(\underline{L_{0_{M'}}})+a_g(\underline{c_0L_{0_{M'}}})\right)
			= \max_{g\in\mathbb{Z}_{m'}}\frac{2}{m'}\sum_{h\in\frac{m'}{2}\mathbb{Z}_{m'}}\psi_h(\underline{L_{0_{M'}}})\psi_{h}(g^{-1})
			\leq \frac{1}{m'}\cdot4\cdot\frac{m'}{2r}=\frac{2}{r},
		\end{equation*}		
		which is impossible because $L_{0_{M'}}\neq \varnothing$ and $r\geq 3$.
		
		We complete the proof.
	\end{proof}
	
	\begin{prop}\label{prop::an2}
		Let $G$  be a abelian group, and let $p$ be an odd prime. Assume that $\Gamma=\mathrm{Cay}(G \oplus \mathbb{Z}_p, S)$ is an antipodal bipartite distance-regular Cayley graph with diameter $4$ in which the antipodal class and the bipartition set containing the identity vertex are $S_0\cup S_4= (0_G, \mathbb{Z}_p)$ and $S_0\cup S_2\cup S_4=H\oplus \mathbb{Z}_p$, respectively, where $H$ is an index $2$  subgroup of $G$.
		Then the following statements hold.
		\begin{enumerate}[(i)]
			\item The sets $R_i$, for $i \in \mathbb{Z}_p$, form a partition of $G\setminus H$.
			\item For every non-trivial character $\psi \in \widehat{\mathbb{Z}_p}$, $B=\{g\in G\mid (\chi_g,\psi)(\underline{S})=\sqrt{k}\}$ is a non-empty set such that
			\begin{equation*}
				\chi_l(\underline{B}) = \frac{|G|}{2\sqrt{k}}\left( \sum_{i \in \mathbb{Z}_p} \psi(i) a_{-l}(\underline{R_i})+\sqrt{k} \cdot a_{-l}(0_G) \right) ~\mbox{for all $l\in G$}.
			\end{equation*}
			\item $\frac{|G|}{2\sqrt{k}}$ is an integer.
		\end{enumerate}
	\end{prop}
	
	\begin{proof}
		(i) As in Proposition \ref{prop::an1}, from Lemma \ref{lem::DR} we can deduce that $\sum_{i\in \mathbb{Z}_p}\underline{R_i}=\underline{G\setminus H}$. Thus the sets $R_i$, for $i \in \mathbb{Z}_p$, form a partition of $G\setminus H$.
		
		(ii) Again by Lemma  \ref{lem::DR}, we have
		\begin{equation}\label{equ::key2caseA1}
			\underline{S}^2=k\cdot e+\mu\underline{S_2}=k\cdot e+\mu\left(\underline{H\oplus \mathbb{Z}_p}-\underline{(0_G,\mathbb{Z}_p)}\right).
		\end{equation}
		Let $\psi\in\widehat{\mathbb{Z}_p}$ be a non-trivial character of $\mathbb{Z}_p$, and let $\chi \in \widehat{G}$. By applying the character $(\chi,\psi)\in \widehat{G\oplus \mathbb{Z}_p}$ on both sides of  \eqref{equ::key2caseA1},  we obtain 
		\begin{equation*}
			((\chi,\psi)(\underline{S}))^2=k,
		\end{equation*}
		implying that $(\chi,\psi)(\underline{S})=\sqrt{k}$ or $-\sqrt{k}$. Let $B=\{g\in G\mid (\chi_g,\psi)(\underline{S})=\sqrt{k}\}$. By a similar analysis as in  Proposition \ref{prop::an1}, we can deduce that 
		\begin{equation*}
			\chi_l(\underline{B}) = \frac{|G|}{2\sqrt{k}}\left( \sum_{i\in \mathbb{Z}_p} \psi(i) a_{-l}(\underline{R_i}) +\sqrt{k} \cdot a_{-l}(0_G) \right)~\mbox{for $l \in G$}.
		\end{equation*}
		In particular, $|B|=\chi_0(\underline{B})=\frac{|G|}{2}$, and so $B$ is non-empty.
		
		(iii) Combining  (i) and (ii), we get
		\begin{equation*}
			\chi_l(\underline{B}^p)=\left(\frac{|G|}{2\sqrt{k}}\right)^p\left(a_l(\underline{G\setminus H})+\sqrt{k^p}\cdot a_l(0_G)\right)~\mbox{for $l \in G$}.
		\end{equation*}
		Let $\sigma:G\rightarrow \mathbb{C}$ be the mapping defined by
		\begin{equation*}
			\sigma(g)=
			\begin{cases} 
				1, & \mbox{if}~g \in H; \\
				-1, & \mbox{if}~g \in G\setminus H.
			\end{cases}   
		\end{equation*}
		As $H$ is an index $2$ subgroup of $G$, the mapping $\sigma$ is exactly an irreducible representation of $G$, and so  $\sigma\in\widehat{G}$ because $G$ is abelian. Thus we assert that there exists some involution $a\in G$ such that $\sigma=\chi_a\in \widehat{G}$. Then from the Fourier inversion formula \eqref{equ::FI} we obtain 
		\begin{equation*}
			\underline{B}^p=\left(\frac{|G|}{2\sqrt{k}}\right)^p\left(\frac{1}{2}\cdot 0_G-\frac{1}{2}\cdot a+\frac{\sqrt{k^p}}{|G|}\underline{G}\right).
		\end{equation*}
		Therefore, $\frac{|G|}{2\sqrt{k}}$ is an integer because $p$ is odd and $\underline{B}^p\in \mathbb{Z} G$.
	\end{proof}

	\section{Distance-regular Cayley graphs over $\mathbb{Z}_{n}\oplus\mathbb{Z}_{p}$}\label{section::7}
	
	In this section, we shall prove Theorem \ref{thm::main}, which determines all distance-regular Cayley graphs over the group $\mathbb{Z}_{n} \oplus \mathbb{Z}_{p}$. To achieve this goal, we need the following two lemmas.

	\begin{lemma}\label{lem::key1}
		Let $p$ be an odd prime with $p\mid n$. There are no antipodal non-bipartite distance-regular Cayley graphs with diameter $3$ over $\mathbb{Z}_{n}\oplus\mathbb{Z}_{p}$.
	\end{lemma}
	\begin{proof}
		By contradiction, assume that $\Gamma = \mathrm{Cay}(\mathbb{Z}_n \oplus \mathbb{Z}_p, S)$ is an antipodal non-bipartite distance-regular Cayley graph of diameter $3$ over $\mathbb{Z}_n \oplus \mathbb{Z}_p$ with $n$ as small as possible. Let $k$ and $r$ ($r \geq 2$) denote the valency and the common size of antipodal classes (or fibres) of $\Gamma$, respectively. According to Lemma \ref{lem::antipodal_DRG}, $k + 1 = \frac{np}{r}$, $k=\mu(r-1)+\lambda+1$, and $\Gamma$ has the intersection array $\{k, \mu(r - 1), 1; 1, \mu, k\}$
		and eigenvalues $k$, $\theta_1$, $\theta_2 = -1$, $\theta_3$, where
		\begin{equation}
			\label{equ::1.3}
			\theta_1 = \frac{\lambda - \mu}{2} + \delta,~~\theta_3 = \frac{\lambda - \mu}{2} - \delta~~\mbox{and}~~\delta = \sqrt{k + \left(\frac{\lambda - \mu}{2}\right)^2}.
		\end{equation}
		Let $H=S_3\cup\{(0,0)\}$ denote the antipodal class containing the identity vertex. Then $|H|=r$. By Lemma \ref{lem::DRG_dq}, $H$ is a subgroup of $\mathbb{Z}_{n}\oplus\mathbb{Z}_{p}$. If $r$ is not  prime, then $H$ has  a non-trivial subgroup $K$. Let $\mathcal{B}$ denote the partition of $\mathbb{Z}_{n}\oplus\mathbb{Z}_{p}$ consisting of all cosets of $K$ in $\mathbb{Z}_{n}\oplus\mathbb{Z}_{p}$, and let $\Gamma_{\mathcal{B}}$ be the quotient graph of $\Gamma$ with respect to $\mathcal{B}$. Then, by a similar way as in Lemma \ref{lem::DRG_dq}, we can verify that  $\Gamma_\mathcal{B}\cong\mathrm{Cay}((\mathbb{Z}_{n}\oplus\mathbb{Z}_{p})/K,S/K)$, where $S/K=\{sK\mid s\in S\}$.  Since $K\cap(S_1\cup S_2)=\varnothing$, for any two distinct  $s_1,s_2\in S$, we have  $s_1K\neq s_2K$.  
		Also, by Corollary \ref{cor::DR}, 
		\begin{equation}\label{equ::quo2}
			\left\{\begin{aligned}
				&\underline{S}^2=k\cdot 0_G+(\lambda-\mu)\underline{S}+\mu(\underline{\mathbb{Z}_{n}\oplus\mathbb{Z}_{p}}-\underline{H}),\\
				&\underline{H}\cdot(\underline{S}+e)=\underline{\mathbb{Z}_{n}\oplus\mathbb{Z}_{p}}.\\ \end{aligned}\right.
		\end{equation}
		Let $f$ be the mapping from the group algebra $\mathbb{Z}\cdot(\mathbb{Z}_{n}\oplus\mathbb{Z}_{p})$ to the group algebra $\mathbb{Z}\cdot((\mathbb{Z}_{n}\oplus\mathbb{Z}_{p})/K)$ defined by 
		\begin{equation*}
			f\left(\sum_{x\in \mathbb{Z}_{n}\oplus\mathbb{Z}_{p}} a_xx\right)=\sum_{x\in \mathbb{Z}_{n}\oplus\mathbb{Z}_{p}}a_x\cdot xK.
		\end{equation*}
		By applying $f$ on both sides of the two equations in \eqref{equ::quo2}, we obtain 
		\begin{equation*}
			\left\{\begin{aligned}
				&(\underline{S/K})^2=k\cdot K+(\lambda-\mu)\underline{S/K}+\mu|K|(\underline{(\mathbb{Z}_{n}\oplus\mathbb{Z}_{p})/K}-\underline{H/K}),\\
				&|K|\underline{H/K}\cdot(\underline{S/K}+K)=|K|\underline{(\mathbb{Z}_{n}\oplus\mathbb{Z}_{p})/K},\\ \end{aligned}\right.
		\end{equation*}
		or equivalently,
		\begin{equation}\label{equ::quo1}
			\left\{\begin{aligned}
				&(\underline{S/K})^2=k\cdot K+((\lambda-\mu+\mu|K|)-\mu|K|)\underline{S/K}+\mu|K|(\underline{(\mathbb{Z}_{n}\oplus\mathbb{Z}_{p})/K}-\underline{H/K}),\\
				&\underline{H/K}\cdot(\underline{S/K}+K)=\underline{(\mathbb{Z}_{n}\oplus\mathbb{Z}_{p})/K}.\\ \end{aligned}\right.
		\end{equation}
		Then from \eqref{equ::quo1} and Corollary \ref{cor::DR} we can deduce that  $\Gamma_\mathcal{B}$ is an ($r/|K|$)-antipodal distance-regular graph of diameter $3$ with intersection array $\{k,k-(\lambda-\mu+\mu|K|)-1=\mu|K|(r/|K|-1),1;1,\mu|K|,k\}$.   If $\Gamma_\mathcal{B}$ is bipartite, then $\Gamma$ is also bipartite, a contradiction. Hence, $\Gamma_\mathcal{B}$ is an antipodal non-bipartite distance-regular Cayley graph of diameter $3$ over the cyclic group or the group  $\mathbb{Z}_{n'}\oplus\mathbb{Z}_{p}$ with $n'\mid n$. By Theorem \ref{thm::cir_DRG}, we assert that the former case cannot occur. For the later case,  this violates the  minimality  of $n$. Therefore, $r$ is a prime. Then, by Proposition \ref{prop::prime_r_antipodal}, $\mathbb{Z}_n\oplus \mathbb{Z}_p\cong M \oplus \mathbb{Z}_r$ and $\Gamma$ is isomorphic to 
		a Cayley graph  over $M \oplus \mathbb{Z}_r$ in which the antipodal class containing the identity vertex is $S_3 \cup \{(0_M,0)\} = (0_M,\mathbb{Z}_r)$, where $M$ is an abelian group of order $|G|/r$. Thus we only need to consider the following two cases.
		
		{\flushleft \bf Case A.} $M=\mathbb{Z}_n$, $r=p$ and $S_0\cup S_3=(0_M,\mathbb{Z}_p)$. 
		
		In this situation, $r=p$ is odd. By Proposition \ref{prop::an1} (ii), there exists a non-empty $(n,-\frac{n}{2\delta}\theta_3, x^p -\frac{n^p}{(2\delta)^p})$-polynomial addition set $B$ in $\mathbb{Z}_n$. Note that $|B|=-\frac{n}{2\delta}\theta_3$.  On the other hand, by Lemma \ref{lem::PAS1}, we assert that $|B|\in \{1,n-1,n\}$. If $|B|=-\frac{n}{2\delta}\theta_3=1$, then from \eqref{equ::1.3} and $k=n-1$ we can deduce that $\lambda-\mu=n-2=k-1$, which is impossible because $k = \mu(p - 1) + \lambda + 1\geq 2\mu+\lambda+1$ and $\mu\geq 1$. Similarly, if $|B|=-\frac{n}{2\delta}\theta_3=n-1$ then  $\mu-\lambda=n-2=k-1$, and if  $|B|=-\frac{n}{2\delta}\theta_3=n$ then $k=0$, which are also impossible.

		{\flushleft \bf Case B.} $M=\mathbb{Z}_{\frac{n}{r}}\oplus\mathbb{Z}_p$ and $S_0\cup S_3=(0_M,\mathbb{Z}_r)$.
		
		In this situation, we must have $\mathrm{gcd}(r,\frac{n}{r})=1$. 
		
		{\flushleft \bf Subcase B.1.} $r=2$.

		Since $p$ is odd and $\mathrm{gcd}(2,\frac{n}{2})=1$, we see that $|M|=\frac{np}{2}$ is odd.	By Proposition \ref{prop::an1} (iii), there exists a strongly regular Cayley graph $\Gamma'$ over $M$ with parameters $(|M|,k'=\frac{|M|}{2\delta}\theta, \lambda'=\frac{|M|}{4\delta^2}(\theta^2-1),\mu'=\frac{|M|}{4\delta^2}(\theta^2-1))$, where $\theta=\theta_1$ or $-\theta_3$.  Clearly, $\Gamma'$ is non-bipartite because it is of odd order.
		If $k'=2$, then $\Gamma'$ is a cycle, and hence $\Gamma'\cong C_4$, which is impossible. Now suppose $k'\geq 3$. We see that  $\Gamma'$  must be primitive or antipodal. If $\Gamma'$ is antipodal, then it is a complete multipartite graph, which is impossible because $\lambda'=\mu'$. If $\Gamma'$ is primitive, then  Corollary \ref{cor::pri_DRG} indicates that $\frac{n}{2}=p$ and $M=\mathbb{Z}_p\oplus \mathbb{Z}_p$ because $\Gamma'$ cannot not a complete graph. Thus, by Lemma \ref{lem::1},   $\Gamma'$ is isomorphic to the line graph of a transversal design $TD(r',p)$ with $2\leq r'\leq p-1$. However, this is also impossible because $\lambda'=\mu'$.

		{\flushleft \bf Subcase B.2.} $r\neq 2$. 
		
		If $r=p$, then we are done by Case A.	Now suppose  $r\neq p$. Recall that  $S=S_1=\cup_{i\in\mathbb{Z}_r}(R_i,i)$. By Proposition \ref{prop::an1} (i), the sets $R_i$, for $i \in \mathbb{Z}_r$, form a partition of $M \setminus \{0_M\}$. Furthermore, by Proposition \ref{prop::an1} (ii), both $\frac{|M|}{2\delta}$ and $\theta_3$ are integers, and for every non-trivial character $\psi\in\widehat{\mathbb{Z}_r}$,  there exists a non-empty polynomial addition set  $B\subseteq M$ such that
		\begin{equation}\label{equ::caseb2}
			\chi_{l}(\underline{B})=\frac{|M|}{2\delta}\left\{\sum_{i\in \mathbb{Z}_r}\psi(i)a_{-l}(\underline{R_i})-\theta_3\cdot a_{-l}(0_M)\right\}~\mbox{for all $l\in M$}.
		\end{equation}
		Let $l_0\in M\setminus\{0_M\}$. Then there exists some $i_0\in \mathbb{Z}_r\setminus\{0\}$ such that $-l_0\in R_i$, and   \eqref{equ::caseb2} implies that
		\begin{equation}\label{equ::caseb2_1}
			\chi_{l_0}(\underline{B})=\frac{|M|}{2\delta}\psi(i_0).
		\end{equation}
		Since $r$ is an odd prime and $\psi\in \widehat{\mathbb{Z}_r}$ is non-trivial, we assert that $\psi(i_0)\in \mathbb{Q}(\omega_r)\setminus\mathbb{Q}$, where  $\omega_r$ is a primitive $r$-th root of unity. Thus it follows from \eqref{equ::caseb2_1}  and $\frac{|M|}{2\delta}\in \mathbb{Z}$ that $\chi_{l_0}(B)\in \mathbb{Q}(\omega_r)\setminus\mathbb{Q}$. On the other hand, we have $\chi_{l_0}(B)\in \mathbb{Q}(\omega_{\frac{n}{r}})$ because $\chi_{l_0}\in \widehat{M}=\widehat{\mathbb{Z}_{\frac{n}{r}}\oplus\mathbb{Z}_p}$ and $p\mid\frac{n}{r}$ due to $p\mid n$ and $p\neq r$, where $\omega_{\frac{n}{r}}$ is a primitive $\frac{n}{r}$-th root of unity. 
		Hence,  $\chi_{l_0}(B)\in (\mathbb{Q}(\omega_{\frac{n}{r}})\cap \mathbb{Q}(\omega_r))\setminus \mathbb{Q}$. However, this is impossible because $\mathbb{Q}(\omega_r)\cap \mathbb{Q}(\omega_{\frac{n}{r}})=\mathbb{Q}$ due to $\mathrm{gcd}(r,\frac{n}{r})=1$.		
		
		Therefore, we conclude that there are no antipodal non-bipartite distance-regular graphs with diameter $3$ over $\mathbb{Z}_{n}\oplus\mathbb{Z}_{p}$.
	\end{proof}
	
	\begin{lemma}\label{lem::key2}
		Let $p$ be an odd prime with $p\mid n$. There are no antipodal bipartite distance-regular Cayley graphs with diameter $4$ over $\mathbb{Z}_{n}\oplus\mathbb{Z}_{p}$.
	\end{lemma}
	\begin{proof}
		By contradiction, assume that  $\Gamma=\mathrm{Cay}(\mathbb{Z}_{n}\oplus\mathbb{Z}_{p},S)$ is an antipodal bipartite distance-regular Cayley graph with diameter $4$. Then $n$ is even and the bipartition set of $\Gamma$ containing the identity vertex is $S_0\cup S_2 \cup S_4=2\mathbb{Z}_n\oplus\mathbb{Z}_p$. Let $k$ and $r$ ($r\geq 2$) denote the valency and the common size of antipodal classes (or fibres) of $\Gamma$, respectively. 
		By Lemma \ref{lem::antipodal_DRG} (ii),
		\begin{equation}\label{equ::key2.1}
			np=2r^2\mu~~\mbox{and}~~k=r\mu,
		\end{equation}
		and $\Gamma$ has the intersection array $\{r\mu, r\mu-1,(r-1)\mu, 1;1,\mu,r\mu-1,r\mu\}$. Moreover, by Lemma \ref{lem::DR},  
		\begin{equation}\label{equ::key2.2}
			\underline{S}^2=k\cdot e+\mu\underline{S_2}=k\cdot 0+\mu(\underline{2\mathbb{Z}_n\oplus\mathbb{Z}_p}-\underline{S_0\cup S_4}).
		\end{equation}
		Note that $S_0\cup S_4$ is the antipodal class of $\Gamma$ containing the identity vertex, and so is a subgroup of $S_0\cup S_2\cup S_4=2\mathbb{Z}_n\oplus \mathbb{Z}_p$. Since $S$ is inverse closed, from \eqref{equ::key2.1} and \eqref{equ::key2.2}	we see that   $S$ is exactly an  $(r\mu,r,r\mu,\mu)$-relative difference set relative to $S_0\cup S_4$ in $S_0\cup S_2\cup S_4=2\mathbb{Z}_n\oplus\mathbb{Z}_p$. Then from Lemma \ref{lem::NRDS} we can deduce that  $\frac{n}{2}=\frac{r^2\mu}{p}$ is a divisor of $r\mu$, that is, $r=p$. Furthermore, by Proposition \ref{prop::anti_bipart}, we may assume that $S_0\cup S_4=(0,\mathbb{Z}_p)$. In this context, by Proposition \ref{prop::an2}, the sets $R_i$, for $i\in \mathbb{Z}_p$, form a partition of $\mathbb{Z}_n\setminus 2\mathbb{Z}_n=1+2\mathbb{Z}_n$, and for every non-trivial character $\psi\in \widehat{\mathbb{Z}_p}$, there exists a non-empty set $B$ in $\mathbb{Z}_n$ such that
		\begin{equation}\label{equ::key2.3}
			\chi_l(\underline{B}) = \frac{n}{2\sqrt{k}}\left( \sum_{i \in \mathbb{Z}_p} \psi(i) a_{-l}(\underline{R_i}) +\sqrt{k} \cdot a_{-l}(0) \right)
			~\mbox{for all $l \in \mathbb{Z}_n$}.	\end{equation}
		Moreover, we have  $\frac{n}{2\sqrt{k}}\in\mathbb{Z}$. Clearly, $\frac{n}{2\sqrt{k}}\neq 1$ by \eqref{equ::key2.1}. Let $q$ be a prime dividor of $\frac{n}{2\sqrt{k}}$. Since the sets $R_i$, for $i\in \mathbb{Z}_p$, form a partition of $1+2\mathbb{Z}_n$, we can deduce from  \eqref{equ::key2.3} that  $\chi_l(\underline{B})\equiv 0\pmod q$ for all $l\in \mathbb{Z}_n$. Then, by Lemma \ref{lem::Ma}, there exist some $\underline{X_1},\underline{X_2}\in \mathbb{Z}\cdot\mathbb{Z}_n$ with non-negative cofficients only such that 	\begin{equation*}
			\underline{B}=q\underline{X_1}+\underline{\frac{n}{q}\mathbb{Z}_n}\cdot\underline{X_2}.
		\end{equation*}
		Since $q>1$ and the  cofficients of $\underline{B}$ in $\mathbb{Z}\cdot \mathbb{Z}_n$ is either $0$ or $1$, we assert that 
		\begin{equation*}
			\underline{B}=\underline{\frac{n}{q}\mathbb{Z}_n}\cdot\underline{X_2}.
		\end{equation*}
		Note that $(1+2\mathbb{Z}_n)\setminus q\mathbb{Z}_n\neq \varnothing$. Taking $l_0\in (1+2\mathbb{Z}_n)\setminus q\mathbb{Z}_n$, we have   $\chi_{l_0}(\underline{B})=\chi_{l_0}(\underline{\frac{n}{q}\mathbb{Z}_n})\cdot \chi_{l_0}(\underline{X_2})=0$. Let $i_0\in\mathbb{Z}_p$ be such  that $-l_0\in R_{i_0}$. Then   \eqref{equ::key2.3} gives that  $\chi_{l_0}(\underline{B})=\frac{n}{2\sqrt{k}}\psi(i_0)$, and hence $\psi(i_0)=0$, which is impossible. 
		
		Therefore,  we conclude that there are no antipodal bipartite distance-regular Cayley graphs with diameter $4$ over $\mathbb{Z}_{n}\oplus\mathbb{Z}_{p}$.
	\end{proof}
	
	Now we are in a position to give the proof of Theorem \ref{thm::main}.
	
	\renewcommand\proofname{\it{Proof of Theorem \ref{thm::main}}}
	\begin{proof}
		First of all, it is easy to verify that the graphs listed in (i)-(iii) are  distance-regular Cayley graphs over $\mathbb{Z}_{n}\oplus\mathbb{Z}_{p}$. Furthermore, by Lemma \ref{lem::1}, the graph listed in (iv) is a distance-regular graph with diameter $2$. 
		
		Conversely, suppose that $\Gamma=\mathrm{Cay}(\mathbb{Z}_{n}\oplus\mathbb{Z}_{p},S)$ is a  distance-regular Cayley graph over $\mathbb{Z}_{n}\oplus\mathbb{Z}_{p}$. If $n=p$, then  from Lemma \ref{lem::1} we see that $\Gamma$ is isomorphic to one of the graphs listed in (i), (ii) and (iv). Thus we may assume that $n\neq p$.  If $\Gamma$ is primitive, by Corollary \ref{cor::pri_DRG}, $\Gamma$ is isomorphic to the complete graph $K_{np}$, as desired. Now suppose that $\Gamma$ is imprimitive. Clearly, $\Gamma$ cannot be isomorphic to a cycle because it is a Cayley graph over $\mathbb{Z}_{n}\oplus\mathbb{Z}_{p}$. Thus $k\geq 3$, and it suffices to consider the  following three situations.
		
		{\flushleft \bf Case A.} $\Gamma$ is antipodal but not bipartite.
		
		By Lemma \ref{lem::imprimitive} and Lemma \ref{lem::DRG_dq}, the antipodal quotient $\overline{\Gamma}$ of $\Gamma$ is a primitive distance-regular  Cayley graph over the cyclic group or the group $\mathbb{Z}_{n'}\oplus\mathbb{Z}_{p}$ for some $n'\mid n$. Then it follows from Theorem \ref{thm::cir_DRG}, Corollary \ref{cor::pri_DRG} and Lemma \ref{lem::1} that $\overline{\Gamma}$ is a complete graph, a cycle of prime order, a Payley graph of prime order, or the line graph of a transversal design $TD(r,p)$ with $2\leq r\leq p-1$.  If $\overline{\Gamma}$ is a cycle of prime order, then $\Gamma$ would be a cycle, which is impossible. If $\overline{\Gamma}$ is a Payley graph of prime order, by  Lemma \ref{lem::confer}, we also deduce a contradiction. If $\overline{\Gamma}$ is the line graph of a transversal design $TD(r,p)$ with $2\leq r\leq p-1$,  then $d=4$ or $5$. By Lemma \ref{lem::antipodal cover of line graph}, we assert that $d=4$ and $r=2$, and hence  $\overline{\Gamma}$ is the Hamming graph $H(2,p)$. However, by Lemma \ref{lem::anti_Ham}, $H(2,p)$ has no distance-regular antipodal covers for $p>2$, and we obtain a contradiction. 
		Therefore, $\overline{\Gamma}$ is a complete graph, and so $d=2$ or $3$. By Lemma \ref{lem::key1}, $d\neq 3$, whence $d=2$. Since complete multipartite graphs are the only antipodal distance-regular graphs with diameter $2$, we conclude that  $\Gamma$ is  a complete multipartite graphs with at least three parts.
		
		{\flushleft \bf Case B.} $\Gamma$ is antipodal and bipartite.
		
		In this situation, $n$ is even. If $d$ is odd, by Lemma \ref{lem::imprimitive}, $\overline{\Gamma}$ is primitive. Also, by  Lemma \ref{lem::DRG_dq},  $\overline{\Gamma}$ is a distance-regular Cayley graph over the cyclic group or the group  $\mathbb{Z}_n'\oplus\mathbb{Z}_p$ for some $n'\mid n$. 
		As in Case A, we assert that $\overline{\Gamma}$ is a complete graph. Hence, $d=3$. Considering that $\Gamma$ is antipodal and bipartite, we obtain  $\Gamma\cong K_{\frac{np}{2},\frac{np}{2}}-\frac{np}{2}K_2$. Moreover, we assert that $n/2$ must be odd, i.e., $n\equiv 2\pmod 4$,  since $\Gamma$ is a Cayley graph over $\mathbb{Z}_n\oplus \mathbb{Z}_p$.  Now suppose that $d$ is even. Then Lemma \ref{lem::imprimitive} and Lemma \ref{lem::DRG_dq} indicate that $\frac{1}{2}\Gamma$ is  an antipodal non-bipartite distance-regular Cayley graph  over  $\mathbb{Z}_{\frac{n}{2}}\oplus\mathbb{Z}_p$ with diameter $d_{\frac{1}{2}\Gamma}=d/2$.  Clearly, $d\neq 2$. By Lemma \ref{lem::key2}, $d\neq 4$.  
		Thus $d\geq 6$ and   $d_{\frac{1}{2}\Gamma}=d/2\geq 3$.  However, this is impossible by Case A.

		{\flushleft \bf Case C.} $\Gamma$ is bipartite but not antipodal.
		
		In this situation, $n$ is even. By Lemma \ref{lem::imprimitive} and Lemma \ref{lem::DRG_dq}, $\frac{1}{2}\Gamma$ is a primitive distance-regular Cayley graph over $\mathbb{Z}_{\frac{n}{2}}\oplus\mathbb{Z}_p$. As in Case A,  $\frac{1}{2}\Gamma$ is a complete graph or the line graph of a transversal design $TD(r,p)$ with $2\leq r\leq p-1$.  In the former case, we have $d=2$ or $3$. If $d=2$, then $\Gamma$ is a complete bipartite graph, as desired. If $d=3$, then  $\Gamma$ is a  non-antipodal bipartite distance-regular graph with diameter $3$ over the abelian group $\mathbb{Z}_n\oplus \mathbb{Z}_p$. By Proposition \ref{prop::non-antipodal_bipartite}, the dual graph $\widehat{\Gamma}$ of $\Gamma$ is an antipodal non-bipartite distance-regular graph with diameter $3$ over $\mathbb{Z}_n\oplus \mathbb{Z}_p$. However, there are no such graphs by  Lemma \ref{lem::key1}, and we obtain a contradiction. In the latter case, we assert that $\mu=1$ by Lemma \ref{lem::key0}. If there exist two distinct elements $a,b\in S$ such that $-a\neq b$, then $0,a,a+b,b,0$ would lead to a cycle of order $4$ in $\Gamma$, and hence $\mu\geq 2$, a contradiction. Thus $S=\{a,-a\}$ for some $a\in \mathbb{Z}_n\oplus \mathbb{Z}_p$, and we see that $\Gamma$ is a cycle, which is also impossible.
		
		We complete the proof.
	\end{proof}

	\section*{Declaration of competing interest}

The authors declare that they have no known competing financial interests or personal
relationships that could have appeared to influence the work reported in this paper.

\section*{Data availability}

No data was used for the research described in the article.

	\section*{Acknowledgements}

The authors would like to thank Professor \v{S}tefko Miklavi\v{c} for many helpful suggestions. X. Huang was supported by National Natural Science Foundation of China (Grant No. 12471324) and Natural Science Foundation of Shanghai (Grant No. 24ZR1415500). L. Lu was supported by National Natural Science Foundation of China (Grant Nos. 12371362, 12001544) and  Natural Science Foundation of Hunan Province  (Grant No. 2021JJ40707).


\begin{thebibliography}{99}\setlength{\itemsep}{0pt}
		
		\bibitem{ADJ17} A. Abdollahi, E. R. van Dam, M. Jazaeri, Distance-regular Cayley graphs with least eigenvalue $-2$, Des. Codes Cryptogr. 84 (2017) 73--85.
		
		\bibitem{ICG} R. C. Alperin, B. L. Peterson, Integral sets and Cayley graphs of finite groups, Electron. J. Combin. 19(1) (2012) \#P44.
		
		\bibitem{Babai} L. Babai, Spectra of Cayley graphs, J. Combin. Theory Ser. B 27 (1979) 180--189.
		
		\bibitem{Bannai} E. Bannai, T. Ito, Algebraic Combinatorics I: Association Schemes,  Benjamin/Cumming, Menlo Park, 1984. 
		
		
		\bibitem{BCN89} A. E. Brouwer, A. M. Cohen, A. Neumaier, Distance-regular Graphs, Springer-Verlag Berlin Heidelberg, New York, 1989.
		
		\bibitem{BV22} A. E. Brouwer, H. Van Maldeghem, Strongly Regular Graphs, Cambridge University Press, Cambridge, 2022.
		
		\bibitem{CTZ11} 	Y. M. Chee, Y. Tan, X. D. Zhang, Strongly regualr graphs constructed from $p$-ary bent functions, J. Algebraic Combin. 34 (2011) 251--266.
		
		
		
		\bibitem{D73}	P. Delsarte, An algebraic approach to the association schemes of coding theory, Philips Res. Rep. Suppl. 10 (1973) vi+97 pp.
		
		\bibitem{ER63} P. Erd\H{o}s,  A. R\'{e}nyi, Asymmetric graphs, Acta Math. Acad. Sci. Hung.  14(3--4) (1963) 295--315.
		
		
		\bibitem{FMX15} T. Feng, K. Momihara, Q. Xiang, Constructions of strongly regular Cayley graphs and skew Hadamard difference sets from cyclotomic classes, Combinatorica 35(4) (2015) 413--434.
		
		\bibitem{FX12} T. Feng, Q. Xiang, Strongly regular graphs from union of cyclotomic classes, J. Combin. Theory Ser. B 102 (2012) 982--995.
		
		
		
		\bibitem{GXY13}  G. Ge, Q. Xiang, T. Yuan, Constructions of strongly regular Cayley graphs using index four Gauss sums, J. Algebraic Combin. 37 (2013) 313--329.
		
		
		
		\bibitem{H86} J. Hemmeter, Distance-regular graphs and halved graphs, European J. Combin. 7 (1986) 119--129.
		
		\bibitem{HD22}  X. Huang, K. C. Das, On distance-regular Cayley graphs of generalized dicyclic groups, Discrete Math. 345 (2022) 112984.
		
		\bibitem{HDL23}  X. Huang, K. C. Das, L. Lu, Distance-regular Cayley graphs over dicyclic groups, J. Algebraic Combin. 57 (2023) 403--420.
		
		\bibitem{HLZ23} X. Huang, L. Lu, X. Zhan, Distance-regular Cayley graphs over (pseudo-) semi-dihedral groups, 2023, \href{https://arxiv.org/abs/2311.08128}{https://arxiv.org/abs/2311.08128}.
		
		\bibitem{Jur95} A. Juri\v{s}i\'{c}, Antipodal covers, PhD thesis, University of Waterloo, Ontario, Canada, (1995).
		
		\bibitem{Jur98} A. Juri\v{s}i\'{c}, Antipodal covers of strongly regular graphs, Discrete Math. 182 (1998) 177--189.
		
		\bibitem{KT} R. Kochendorfer, Untersuchungen \"{u}ber eine Vermutung von W. Burnside, Schr. Math. Sem. Inst. Angew. Math. Univ. Berlin 3 (1937) 155--180.
		
		
		\bibitem{LM05}  Y. I. Leifman, M. Muzychuk, Strongly regular Cayley graphs over the group $\mathbb{Z}_{p^n}\oplus\mathbb{Z}_{p^n}$, Discrete Math. 305 (2005) 219--239.
		
		\bibitem{LM95} K. H. Leung, S. L. Ma, Partial difference sets with Paley parameters, Bull. London Math. Soc. 27 (1995) 553--564.
		
		
		\bibitem{M84} S. L. Ma, Partial difference sets, Discrete Math. 52 (1984) 75--89.
		
		\bibitem{MaPhD} S. L. Ma, Polynomial addition sets, PhD thesis, University of Hong Kong, Hong Kong, 1985.
		
		\bibitem{M94}  S. L. Ma, A survey of partial difference sets, Des. Codes Cryptogr. 4 (1994) 221--261.
		
		\bibitem{M89} D. Maru\v{s}i\v{c}, Strong regularity and circulant graphs, Discrete Math. 78 (1989) 119--125.
		
		\bibitem{MP03} \v{S}. Miklavi\v{c}, P. Poto\v{c}nik, Distance-regular circulants, European J. Combin. 24 (2003) 777--784.
		
		\bibitem{MP07} \v{S}. Miklavi\v{c}, P. Poto\v{c}nik,  Distance-regular Cayley graphs on dihedral groups, J. Combin. Theory Ser. B 97 (2007) 14--33.
		
		\bibitem{MS14} \v{S}. Miklavi\v{c}, P. \v{S}parl, On distance-regular Cayley graphs on abelian groups, J. Combin. Theory Ser. B 108 (2014) 102--122.
		
		\bibitem{MS20} \v{S}. Miklavi\v{c}, P. \v{S}parl, On minimal distance-regular Cayley graphs of generalized dihedral groups, Electron. J. Combin. 27(4) (2020) \#P4.33.
		
		
		\bibitem{Mom13} K. Momihara, Strongly regular Cayley graphs, skew Hadamard difference sets, and rationality of relative Gauss sums, European J. Combin. 34 (2013) 706--723.
		
		\bibitem{Mom14} K. Momihara,  Certain strongly regular Cayley graphs on $\mathbb{F}_{2^{2(2s+1)}}$ from cyclotomy, Finite Fields Appl. 25 (2014) 280--292.
		
		\bibitem{Mom17} 	K. Momihara, Construction of strongly regular Cayley graphs based on three-valued Gauss periods, European J. Combin. 70 (2017) 232--250.
		
		\bibitem{MX14}  K. Momihara, Q. Xiang, Lifting constructions of strongly regular Cayley graphs, Finite Fields Appl. 26 (2014) 86--99.
		
		
		
		\bibitem{MX18}  K. Momihara, Q. Xiang, Strongly regular Cayley graphs from partitions of subdifference sets of the Singer difference sets, Finite Fields Appl. 50 (2018) 222--250.
		
		
		\bibitem{MX22} K. Momihara, Q. Xiang, Cyclic arcs of Singer type and strongly regular
		Cayley graphs over finite fields, Finite Fields Appl. 77 (2022) 101953.
		
		\bibitem{MP09} M. Muzychuk, I. Ponomarenko, Schur rings, European J. Combin. 30  (2009) 1526--1539.
		
		\bibitem{RDS} A. Pott, Finite Geometry and Character Theory, Lecture Notes in Mathematics, vol. 1601, Springer, Berlin, 1995.
		
		\bibitem{QCM24} L. Qian, X. Cao, J. Michel,	Constructions of strongly regular Cayley graphs derived from weakly regular bent functions, Discrete Math. 347 (2024) 113916.
		
		
		\bibitem{S02} B. Schmidt, Characters and Cyclotomic Fields in Finite Geometry, Lecture Notes in Mathematics, vol. 1797, Springer, Berlin, 2002.
		
		\bibitem{Ste12} B. Steinberg, Representation Theory of Finite Groups: An Introductory Approach, Springer-Verlag, 2012.
		
		\bibitem{S99} T. Sz\H{o}nyi, Around R\'{e}dei’s theorem, Discrete Math. 208--209 (1999) 557--575.
		
		\bibitem{TPF10}  Y. Tan, A. Pott, T. Feng, Strongly regular graphs associated with ternary bent functions, J. Combin. Theory Ser. A 117(6) (2010) 668--682.
		
		\bibitem{HG} J. T. M. Van Bon, A. E. Brouwer, The distance-regular antipodal covers of classical distance-regular graphs,
		in: Combinatorics (Eger, 1987), Colloqium in Mathematical Society J\'{a}nos Bolyai, Vol. 52,
		North-Holland, Amsterdam, 1988, pp. 141--166.
		
		
		\bibitem{DJ19} E. R. van Dam, M. Jazaeri, Distance-regular Cayley graphs with small valency, Ars Math. Contemp. 17 (2019) 203--222.
		
		\bibitem{DJ21} E. R. van Dam, M. Jazaeri,  On bipartite distance-regular Cayley graphs with small diameter, Electron. J. Combin. 29(2) (2022) \#P2.12.
		
		\bibitem{DK06} E. R. van Dam, J. H. Koolen,  A new family of distance-regular graphs with unbounded diameter, Invent. math. 162 (2005) 189--193.
		
		\bibitem{DKT16} E. R. van Dam, J. H. Koolen, H. Tanaka, Distance-regular graphs, Electron. J. Combin. (2016) DS22.
		
		\bibitem{Wu13} F. Wu, Constructions of strongly regular Cayley graphs using even index Gauss sums, J. Combin. Des. 21(10) (2013) 432--446.
		
		\bibitem{ZLH23} X. Zhan, L. Lu, X. Huang,  Distance-regular Cayley graphs over $\mathbb{Z}_{p^s}\oplus\mathbb{Z}_{p}$, Ars Math. Contemp., 2024, to appear, \href{https://doi.org/10.26493/1855-3974.3242.12b}{https://doi.org/10.26493/1855-3974.3242.12b}.
	\end{thebibliography}
\end{document}